\documentclass[10pt]{article}

\usepackage{amsmath,amssymb,amsthm}
\usepackage[utf8]{inputenc}
\usepackage{geometry}
\usepackage{xcolor}
\usepackage{titling}

\usepackage[backend=biber]{biblatex}

\usepackage[
    colorlinks=true,
    linkcolor=blue,
    citecolor=red,
    urlcolor=blue
]{hyperref}

\numberwithin{equation}{section}

\theoremstyle{plain}
\newtheorem{theorem}{Theorem}[section]
\newtheorem{lemma}[theorem]{Lemma}

\theoremstyle{definition}

\theoremstyle{remark}

\newcommand{\ds}{\,\mathrm{d}s}

\newcommand{\dW}{\,\mathrm{d}W}

\title{
Small-Time Asymptotic Behavior of the Stochastic
Landau--Lifshitz--Baryakhtar Equation
}

\author{
Utpal Manna\thanks{
School of Mathematics, Indian Institute of Science Education and Research Thiruvananthapuram,\\
Vithura, Thiruvananthapuram 695551, India.\\
\textit{E-mail:} manna.utpal@iisertvm.ac.in
}
\and
Sangram Satpathi$^{\sharp}$\thanks{
School of Mathematics, Indian Institute of Science Education and Research Thiruvananthapuram,\\
Vithura, Thiruvananthapuram 695551, India.\\
$^{\sharp}$ Corresponding author.\\
\textit{E-mail:} sangramsatpathi20@iisertvm.ac.in
}
}
\date{}

\begin{document}

\maketitle

\begin{abstract}
We establish a small-time large deviation principle for the stochastic Landau--Lifshitz--Baryakhtar equation using the framework of exponential equivalence. This result characterizes the asymptotic behavior of the solution on very short time scales. In particular, it shows that, as the stochastic thermal fluctuations become small, the magnetization remains exponentially concentrated near its initial state, reflecting the short-time stability of the magnetization dynamics. The associated rate function provides a quantitative measure of deviations from the initial state and the resulting short-time stability.
\end{abstract}
\medskip
\noindent\textbf{MSC:} 35Q56, 35Q60, 60H15, 60F10.

\medskip
\noindent\textbf{Keywords:}
Landau--Lifshitz--Baryakhtar equation; small-time asymptotics;
large deviation principle; stochastic PDEs.
\tableofcontents

\section{Introduction}

\subsection{The Landau--Lifshitz--Baryakhtar equation and magnetization dynamics}
Magnetism of matter is one of the highly active research fields in physics. The mathematical theory of ferromagnetism was initiated by Weiss~\cite{weiss}, who gave the idea of spontaneous magnetization in ferromagnetic materials. In 1935, Landau and Lifshitz~\cite{landaulifshitz} proposed the Landau--Lifshitz equation to describe the dynamics of magnetization in ferromagnetic spin systems. Later, in 1955, Gilbert~\cite{gilbert} further advanced the theory of ferromagnetism and proposed the Landau--Lifshitz--Gilbert (LLG) equation. This equation describes the evolution of the spin magnetic moment in magnetic systems, focusing especially on precession and dissipation under an external magnetic field. More precisely, let the magnetization field be denoted by \(u = u(t,x) \in \mathbb{R}^{3}\), defined on a magnetic domain \(\mathcal{D} \subset \mathbb{R}^{d}\) with \(d = 1,2,3\). Then the Landau--Lifshitz--Gilbert equation is given by
\[
\frac{\partial u}{\partial t}
=
-\lambda_{1} \, u \times H_{\mathrm{eff}}
-\lambda_{2} \, u \times \bigl(u \times H_{\mathrm{eff}}\bigr),
\]
where \(H_{\mathrm{eff}}\) denotes the effective magnetic field, \(\times\) represents the vector cross product in \(\mathbb{R}^{3}\), and the positive constants \(\lambda_{1}\) and \(\lambda_{2}\) correspond, respectively, to the gyromagnetic ratio and the damping parameter. In many mathematical studies, the effective field is taken as \(H_{\mathrm{eff}} = \Delta u\), i.e., the exchange field only.
It is well known that the LLG equation provides a valid description of ferromagnetic dynamics only when the temperature remains below the Curie temperature \(T_{c}\). The model has found numerous applications in the study of magnetic storage devices, nanomagnetic materials, and spintronic systems. From a mathematical viewpoint, the existence, uniqueness, and regularity properties of the LLG equation have already been studied; see, for instance,~\cite{ref1,ref2,ref3,ref4} and the references therein.

However, the LLG equation loses its validity in the high-temperature regime near or above the Curie temperature. To overcome this limitation, Garanin~\cite{garanin} developed a thermodynamically consistent approach and derived the following Landau--Lifshitz--Bloch (LLB) equation, which remains valid both below and above the Curie temperature. The LLB equation can be written as
\[
\frac{\partial u}{\partial t}
=
\gamma\, u \times H_{\mathrm{eff}}
+
\frac{L_{1}}{|u|^{2}}(u\cdot H_{\mathrm{eff}})u
-
\frac{L_{2}}{|u|^{2}}
u\times (u\times H_{\mathrm{eff}}),
\]
where \(\gamma>0\) denotes the gyromagnetic ratio, while \(L_{1}\) and \(L_{2}\) represent the longitudinal and transverse damping parameters, respectively. The effective magnetic field is typically given by
\[
H_{\mathrm{eff}}
=
\Delta u
-
\frac{1}{\chi_{\parallel}}
\left(
1+\frac{3T}{5(T-T_{c})}|u|^{2}
\right)u,
\]
where \(\chi_{\parallel}\) denotes the longitudinal susceptibility.

The LLB equation\cite{ref31} may be considered as an interpolation between the classical LLG dynamics at low temperatures and the Ginzburg--Landau theory of phase transitions at high temperatures. This introduction has significantly enriched our understanding of the behavior of magnetic materials, particularly in the study of magnetic nanoparticles and single-molecule magnets~\cite{ref5}.

Despite their success, both the LLG and LLB equations fail to explain several experimentally observed phenomena, including nonlocal damping effects in magnetic metals and crystals~\cite{ref18,ref43}, as well as the unexpectedly large spin-wave damping for short-wave magnons~\cite{ref4}. To address these limitations, Baryakhtar~\cite{ref2,ref3,ref4} proposed an extension of the classical Landau--Lifshitz type models based on Onsager's relations, leading to the so-called Landau--Lifshitz--Baryakhtar (LLBar) equation.
The general form of LLBar equation ~\cite{ref4,ref42} reads
\[
\frac{\partial u}{\partial t}
=
-\lambda_{1}\, u\times H_{\mathrm{eff}}
+
\Lambda_{r}\cdot H_{\mathrm{eff}}
-
\Lambda_{e,ij}
\frac{\partial^{2} H_{\mathrm{eff}}}{\partial x_{i}\partial x_{j}},
\]
where \(u\) denotes the magnetization vector, while \(\Lambda_{r}\) and \(\Lambda_{e}\) represent the relaxation and exchange tensors, respectively. 
For polycrystalline and amorphous soft magnetic materials, as well as magnetic metals at moderate temperatures where nonlocal damping and longitudinal relaxation are significant, a simplified form of the Landau--Lifshitz--Baryakhtar (LLBar) equation can be written as follows~\cite{ref18,ref42}:
\[
\frac{\partial u}{\partial t}
=
-\lambda_{1}\, u\times H_{\mathrm{eff}}
+
\lambda_{r} H_{\mathrm{eff}}
-
\lambda_{e}\Delta H_{\mathrm{eff}},
\]
where the positive constants \(\lambda_{1}\), \(\lambda_{r}\), and \(\lambda_{e}\) denote the gyromagnetic ratio, the relativistic damping parameter, and the exchange damping coefficient, respectively. The effective magnetic field is usually taken as
\[
H_{\mathrm{eff}}
=
\Delta u
+
\frac{1}{2\chi}(1-|u|^{2})u,
\]
where \(\chi>0\) denotes the magnetic susceptibility of the material.
There are several results available for this LLBar equation; see, for instance,~\cite{ref18,ref42,ref43,ref40}.
To study the stochastic version of the LLBar equation, we incorporate stochastic perturbations into the model in order to obtain a more realistic description of magnetization dynamics in noisy environments. We perturb the effective field by a Gaussian-type noise as follows:
\[
H_{\mathrm{eff}}
\longmapsto
H_{\mathrm{eff}}+\xi,
\]
where \(\xi\) denotes a Gaussian random perturbation. Consequently, the stochastic LLBar equation takes the form
\[
\frac{\partial u}{\partial t}
=
-\lambda_{1}u\times(H_{\mathrm{eff}}+\xi)
+
\lambda_{r}(H_{\mathrm{eff}}+\xi)
-
\lambda_{e}\Delta(H_{\mathrm{eff}}+\xi).
\]

The stochastic LLBar equation has a more complicated structure than the stochastic LLG and stochastic LLB equations due to the presence of higher-order diffusion and nonlinear damping terms. As a result, its mathematical analysis becomes more difficult. We refer to~\cite{fan,xufan, golds} for some recent results on the stochastic LLBar equation.
\subsection{The Stochastic LLBar Equation}
The Gaussian noise considered in this work is given by
\[
\xi(t)
=
\sum_{j=1}^{\infty}
h_{j}\circ \frac{dW_{j}(t)}{dt},
\]
where \(\{W_{j}\}_{j\geq1}\) is a sequence of independent real-valued Wiener processes, and \(\{h_{j}\}_{j\geq1}\) are space-dependent deterministic coefficients, which satisfy the suitable regularity assumptions. The stochastic integral is understood in the Stratonovich sense. Throughout this work, we assume that
\begin{equation}\label{eq:hj_condition}
\sum_{j=1}^{\infty}
\|h_{j}\|_{H^{3}}^{2}
\leq C_{h}<\infty.
\end{equation}

Substituting the effective magnetic field into the stochastic LLBar model yields the following fourth-order stochastic partial differential equation written in Stratonovich form:
\begin{equation}\label{eq:stochastic_LLBar}
\left\{
\begin{aligned}
du &= \Big[ \beta_{1}\Delta u - \beta_{2}\Delta^{2}u + \beta_{3}(1-|u|^{2})u - \beta_{4}u\times \Delta u + \beta_{5}\Delta(|u|^{2}u) \Big]dt \\
&\quad + \sum_{j=1}^{\infty} \bigl( -u\times h_{j} + h_{j} - \Delta h_{j} \bigr) \circ dW_{j}(t), \qquad \text{in } (0,T)\times \mathcal{D}, \\[1mm]
\frac{\partial u}{\partial n} &= \frac{\partial \Delta u}{\partial n} = 0, \qquad \text{on } (0,T)\times \partial\mathcal{D}, \\[1mm]
u(0) &= u_{0}, \qquad \text{in } \mathcal{D}.
\end{aligned}
\right.
\end{equation}

Here, \(\mathcal{D}\subset \mathbb{R}^{d}\), \(d=1,2,3\), is a bounded smooth domain, and \(n\) denotes the outward unit normal vector on \(\partial\mathcal{D}\). The constant
\[
\beta_{1}
=
\lambda_{r}
-
\frac{\lambda_{e}}{2\chi},
\]
may be either positive or negative, whereas \(\beta_{2},\ldots,\beta_{5}\) are positive constants. In general, the constant \(\beta_{1}\) is expected to be positive since \(\lambda_{e}/(2\chi)\) is typically much smaller than \(\lambda_{r}\). However, in certain situations arising in spintronics or magnonics, where the wavelength of magnons approaches the exchange length of the ferromagnetic material, the parameter \(\lambda_{e}\) may become significant~\cite{PhysRevB}. For the simplification of the calculations, we assume that \(\beta_{1}>0\). We also prove all the results for the one-dimensional case \(d=1\). However, it is possible to extend the analysis to the case \(d=2\) and to arbitrary \(\beta_{1}\in\mathbb{R}-\{0\}\) with only minor modifications in the arguments; see~\cite{xufan}.
\subsection{Main Objective and Literature Review}

The main objective of the present work is to study the small-time asymptotic behavior of solutions to the stochastic Landau--Lifshitz--Baryakhtar equation \eqref{eq:stochastic_LLBar}. More precisely, we establish a small-time large deviation principle (LDP) for the corresponding stochastic system under suitable assumptions on the noise and initial data. Large deviation theory concerns the estimation of the probabilities of rare events that decay exponentially. The small-time LDP describes how likely the solution is to remain close to its initial value for very short times.

For modern large deviation theory, the earliest framework was introduced by Varadhan~\cite{V66} in 1966. The study of small-time asymptotics began with Varadhan's work~\cite{V67} in 1967, where he considered finite-dimensional diffusion processes. In that work, the idea was to reduce the small-time LDP problem to a Freidlin--Wentzell type small-noise large deviation problem, and using the concept of exponential equivalence, the authors were able to establish the small-time LDP. Several articles have applied this method; see, for example,~\cite{mann,ROCKNER2012716,group,hilbert,liu,orh,phi,pri,rock,kumar2025globalwellposednesssmalltime}. There are also other works based on this method. For instance, in~\cite{navier}, the authors studied the small-time asymptotics of the two-dimensional Navier--Stokes equation, and in~\cite{llb}, the small-time behavior of the stochastic Landau--Lifshitz--Bloch equation was investigated using the same approach. In this paper, we also follow this method to establish the small-time LDP for our equation. For more techniques related to the large deviation principle, we refer to the book~\cite{book}.

Alternatively, small-time asymptotics can also be studied using the weak convergence method. However, there is limited literature on small-time large deviations for stochastic partial differential equations using this approach. One such work is by Juan Yang et al.~\cite{delay}, where they studied small-time large deviations for a stochastic delay equation. On the other hand, for the small-noise large deviation principle, this method has been widely established for stochastic partial differential equations; see, for example,~\cite{SRITHARAN,akash,Qiu2020, chengg} and the references therein.

Since we aim to study the small-time large deviation principle for the stochastic LLBar equation, let us elaborate a bit more on its physical interpretation. The principle describes the fact that the solution is concentrated around the initial position with exponential probability after a short time. That is, the magnetomechanical properties of ferromagnetic materials are stabilized in the initial state as the stochastic thermal effect gradually diminishes. Moreover, the rate function of the large deviation portrays this relative stability. An important motivation to consider such a problem comes from the Varadhan identity
\[
\lim_{t \to 0} 2t \log \mathbb{P}\bigl(u(0) \in B, u(t) \in C\bigr) = -d^2(B, C),
\]
where \(u\) is the strong solution of the stochastic LLBar equation and \(d\) is an appropriate Riemannian distance associated with the diffusion generated by \(u\). It is worth mentioning that the small-time large deviation principle for finite-dimensional diffusion processes was initiated by Varadhan~\cite{V67}, and the small-time large deviation principle for infinite-dimensional models was studied in~\cite{Aida200267, Hino20031254, navier, dong, ROCKNER2012716, manil} and the references therein.

The layout of the present paper is as follows. In Section \ref{sec:2}, we state some prelimineries and the main result. Section \ref{sec:3} is devoted to establishing the small-time large deviation principle for Eq.~\eqref{eq:stochastic_LLBar}.

\section{Preliminaries and Main Results}{\label{sec:2}}
\subsection{Functional Setting and Auxiliary Facts}
We assume that \(\mathcal{D}\subset \mathbb{R}\) is a bounded open domain. For \(1\leq p\leq \infty\), let \(L^{p}(\mathcal{D}) = L^{p}(\mathcal{D};\mathbb{R}^{3})\) denote the vector-valued Lebesgue space with norm \(\|\cdot\|_{L^{p}}\). The inner product in \(L^{2}(\mathcal{D})\) is \(\langle \cdot,\cdot\rangle\). For \(m\geq 0\), \(H^{m}=H^{m}(\mathcal{D};\mathbb{R}^{3})\) is the usual Sobolev space with norm \(\|u\|_{H^{m}}^{2} = \|(I-\Delta)^{m/2}u\|_{L^{2}}^{2}\) and inner product \(\langle \cdot,\cdot\rangle_{H^{m}}\). Let \(H^{-m}\) be the dual space of \(H^{m}\), with duality pairing \({}_{H^{-m}}\langle \cdot,\cdot\rangle_{H^{m}}\). For simplicity, write \(H := L^{2}(\mathcal{D})\).

Since the spatial domain is one-dimensional, the Sobolev embedding theorem yields \(H^{1}(\mathcal{D}) \hookrightarrow L^{\infty}(\mathcal{D}) \hookrightarrow L^{p}(\mathcal{D})\) for \(1 \leq p < \infty\).

In our analysis, we will frequently use the following identities.
For any vector-valued function \(v : D \to \mathbb{R}^3\), we have
\begin{align}
\nabla \bigl(|v|^2 v\bigr) &= 2v (v \cdot \nabla v) + |v|^2 \nabla v, \label{eq:aux1} \\
\frac{\partial}{\partial n} \bigl(|v|^2 v\bigr) &= 2v \left( v \cdot \frac{\partial v}{\partial n} \right) + |v|^2 \frac{\partial v}{\partial n}, \label{eq:aux2} \\
\Delta \bigl(|v|^2 v\bigr)
&=2|\nabla v|^2 v+2(v \cdot \Delta v)\,v+4\,\nabla v\,(v \cdot \nabla v)^{\top}+|v|^2 \Delta v.
\label{eq:aux3}
\end{align}

Similar to~\cite{xufan}, we can establish the following theorem for equation~\eqref{eq:stochastic_LLBar}.
\begin{theorem}\label{thm:main_result}
Assume that the initial data \(u_{0}\in H^{1}\). Fix a stochastic basis \((\Omega, \mathcal{F}, \mathbb{F}, \mathbb{P}, W)\). Then, for every \(T>0\), the stochastic LLBar equation \eqref{eq:stochastic_LLBar} admits a unique solution
\[
u \in L^{p}\bigl( \Omega; L^{\infty}(0,T;H^{1}) \bigr) \cap L^{2}\bigl( \Omega; L^{2}(0,T;H^{3}) \bigr),
\]
for every \(p\geq 1\). Moreover, for every \(t\in[0,T]\) and every \(\phi\in H^{1}\), the following weak formulation holds \(\mathbb{P}\)-a.s.:
\begin{equation}\label{eq:weak_formulation_main}
\begin{aligned}
\langle u(t), \phi \rangle_H &= \langle u_0, \phi \rangle_H
+ \beta_1 \int_0^t \langle \Delta u(s), \phi \rangle_H \, ds
- \beta_2 \int_0^t \langle \nabla\Delta u(s), \nabla \phi \rangle_H \, ds \\
&\quad + \beta_3 \int_0^t \langle (1-|u(s)|^2) u(s), \phi \rangle_H \, ds
+ \beta_4 \int_0^t \langle u(s) \times \Delta u(s), \phi \rangle_H \, ds \\
&\quad + \beta_5 \int_0^t \langle \Delta(|u(s)|^2 u(s)),  \phi \rangle_H \, ds \\
&\quad + \sum_{j=1}^{\infty} \int_0^t \langle -u(s) \times h_j + h_j - \Delta h_j, \phi \rangle_H \circ dW_j(s).
\end{aligned}
\end{equation}
\end{theorem}
\subsection{Main result}

Let $\varepsilon>0$. Then it is easy to see that the small-time process ${u}(\varepsilon t)$ coincides in law with the solution of the following equation:
\begin{equation}\label{eq:ito_LLBar}
\begin{aligned}
{u}^\varepsilon(t)
={}&\, {u}_0
+\varepsilon\beta_1\int_0^t \Delta{u}^\varepsilon(s)\ds
-\varepsilon\beta_2\int_0^t \Delta^2{u}^\varepsilon(s)\ds 
+\varepsilon\beta_3\int_0^t
\bigl(1-|{u}^\varepsilon(s)|^2\bigr){u}^\varepsilon(s)\ds \\
&-\varepsilon\beta_4\int_0^t
{u}^\varepsilon(s)\times\Delta{u}^\varepsilon(s)\ds 
+\varepsilon\beta_5\int_0^t
\Delta\Bigl(|{u}^\varepsilon(s)|^2{u}^\varepsilon(s)\Bigr)\ds\\
&-\frac{\varepsilon}{2}\sum_{j=1}^\infty
\int_0^t
{{G}_j}(u^\varepsilon)\times {h}_j\ds
+\sqrt{\varepsilon}\sum_{j=1}^\infty
\int_0^t
{{G}_j}(u^\varepsilon)\,\dW_j(s),
\end{aligned}
\end{equation}
where ${{G}_j}(u^\varepsilon)$ is defined as:
\begin{equation}
{{G}_j}(u^\varepsilon) = -{u}^\varepsilon\times {h}_j +  {h}_j - \Delta {h}_j.
\end{equation}

Let $\nu^\varepsilon$ be the law of $u^\varepsilon$ on $C([0,T];H^{1})$.  
Define a functional $I(g)$ on $C([0,T];H^{1})$ by
\begin{equation}{\label{rate}}
I(g)
=
\inf_{l\in \Gamma_g}
\left\{
\frac12
\int_0^T
|l(t)|_{H}^{\,2}\,dt
\right\}.
\end{equation}
where
\[
\Gamma_g
=\left\{
l\in C([0,T];H^{1}) :
\ l \text{ is absolutely continuous and }
g(t)
=x+\sum_{j=1}^{\infty}\int_0^t G_j\bigl(g^\varepsilon(s)\bigr) l(s)\,ds,
\quad 0\le t\le T
\right\}.
\]
Now we state the main result of this paper.

\begin{theorem}{\label{main_theorem}}
The family $\{\nu^\varepsilon\}_{\varepsilon>0}$
satisfies the large deviation principle on
$C([0,T];H^{1})$ with the rate function $I(\cdot)$
given by \eqref{rate}, i.e.,

\begin{enumerate}
\item[(i)] For any closed subset
$F\subset C([0,T];H^{1})$,
\[
\limsup_{\varepsilon\to0}
\varepsilon \log \nu^\varepsilon(F)
\le
-\inf_{g\in F} I(g).
\]

\item[(ii)] For any open subset
$E\subset C([0,T];H^{1})$,
\[
\liminf_{\varepsilon\to0}
\varepsilon \log \nu^\varepsilon(E)
\ge
-\inf_{g\in E} I(g).
\]
\end{enumerate}
\end{theorem}
For simplicity of notation, let \(C\) denote a generic positive constant, which may vary from line to line throughout the rest of the paper. Moreover, we use \(x\) to denote the initial value of both \(u^{\varepsilon}\) and \(v^{\varepsilon}\).
\section{Proof of main result}{\label{sec:3}}
In this section, we establish the proof of Theorem~\ref{main_theorem}, which is mainly based on the method of exponential equivalence.

Let $v^\varepsilon(\cdot)$ be the solution of the following stochastic differential equation:
\begin{equation}\label{eq:Yeps}
v^\varepsilon(t)
=
x+\sqrt{\varepsilon}\sum_{j=1}^\infty
\int_0^t
{{G}_j}(v^\varepsilon)\,\dW_j(s),
\end{equation}
and let $\nu^\varepsilon$ be the law of $v^\varepsilon(\cdot)$ on
$C([0,T];H^{1})$. By \cite[Theorem 12.9]{daprato}, the family
$\{\nu^\varepsilon\}_{\varepsilon>0}$ satisfies a large deviation principle on
$C([0,T];H^{1})$ with rate function $I$ defined in \eqref{rate}; see also
\cite{budhi,liuLDP,tao}. Therefore, to prove
Theorem~\ref{main_theorem}, it is enough to show that the measures
$\mu^\varepsilon$ and $\nu^\varepsilon$ are exponentially equivalent.
Indeed, by \cite[Theorem 4.2.13]{dembo}, it suffices to establish that for every
$\delta>0$,
\begin{equation}\label{exp_equiv}
\lim_{\varepsilon\to0}
\varepsilon
\log
\mathbb{P}
\left(
\sup_{0\le t\le T}
\|u^\varepsilon(t)-v^\varepsilon(t)\|_{H^{1}}^{2}
>\delta
\right)
=-\infty.
\end{equation}
Then Theorem~\ref{main_theorem} follows from the fact that if one of two exponentially equivalent families satisfies the large deviation principle, then so does the other.

For this purpose, we first establish several lemmas.
Let $u^\varepsilon(t)$ be the solution of \eqref{eq:ito_LLBar}. Then the following result holds.
\begin{lemma}\label{lemma1}
Let \(u^\varepsilon(t)\) be the solution of \eqref{eq:ito_LLBar}. Then
\[
\lim_{M\to\infty}
\sup_{0<\varepsilon\le1}
\varepsilon
\log
\mathbb{P}
\left(
\|u^\varepsilon\|_{H^1}^{H^2}(T) > M
\right)
=
-\infty,
\]
where
\[
\|u^\varepsilon\|_{H^1}^{H^2}(T)
:=
\sup_{0\le t\le T}
\|u^\varepsilon(t)\|_{H^1}^{2}
+
2C\varepsilon
\int_0^T
\|u^\varepsilon(s)\|_{H^2}^{2}\, ds,
\]
with \(C>0\).
\end{lemma}
\begin{proof}

Applying It\^o's formula to $\|u^\varepsilon(t)\|_{H}^2$, and using the fact that\(\bigl\langle u^\varepsilon\times\Delta u^\varepsilon,\,u^\varepsilon
\bigr\rangle_H = 0,\)
we obtain
\begin{equation}\label{eq:ito-energy}
\begin{aligned}
\|u^\varepsilon(t)\|_{H}^2
&+2\varepsilon\beta_1\int_0^t \|\nabla u^\varepsilon(s)\|_{H}^2\,ds
+2\varepsilon\beta_2\int_0^t \|\Delta u^\varepsilon(s)\|_{H}^2\,ds
+2\varepsilon\beta_3\int_0^t \|u^\varepsilon(s)\|_{L^4}^4\,ds \\
&=
\|u^\varepsilon(0)\|_{H}^2
+2\varepsilon\beta_3\int_0^t \|u^\varepsilon(s)\|_{H}^2\,ds
+2\varepsilon\beta_5\int_0^t \Bigl\langle \Delta\bigl(|u^\varepsilon(s)|^2u^\varepsilon(s)\bigr), u^\varepsilon(s) \Bigr\rangle_H\,ds \\
&\quad
-\varepsilon\sum_{j=1}^{\infty}\int_0^t \Bigl\langle G_j(u^\varepsilon(s))\times h_j, u^\varepsilon(s) \Bigr\rangle_H\,ds
+\varepsilon\sum_{j=1}^{\infty}\int_0^t \|G_j(u^\varepsilon(s))\|_{H}^2\,ds \\
&\quad
+2\sqrt{\varepsilon}\sum_{j=1}^{\infty}\int_0^t \Bigl\langle G_j(u^\varepsilon(s)), u^\varepsilon(s) \Bigr\rangle_H\,dW_j(s).
\end{aligned}
\end{equation}

Since $\dfrac{\partial u^\varepsilon}{\partial n}=0$ on $\partial\mathcal{D}$, integration by parts yields
\begin{equation}\label{eq:integration-parts}
\begin{aligned}
\Bigl\langle \Delta\bigl(|u^\varepsilon|^2u^\varepsilon\bigr),\; u^\varepsilon \Bigr\rangle_H
&= - \Bigl\langle \nabla\bigl(|u^\varepsilon|^2u^\varepsilon\bigr),\; \nabla u^\varepsilon \Bigr\rangle_H \\
&= -2\|u^\varepsilon\cdot\nabla u^\varepsilon\|_{H}^2 - \bigl\||u^\varepsilon|\,|\nabla u^\varepsilon|\bigr\|_{H}^2 .
\end{aligned}
\end{equation}
By H\"older's inequality and Young's inequality, together with the assumption
\(
\sum_{j=1}^{\infty}\|h_j\|_{H^3}<C_h<\infty,
\)
we obtain
\begin{equation}\label{eq:gj-estimate}
\begin{aligned}
\sum_{j=1}^{\infty}
\bigl|
\bigl(
G_j(u^\varepsilon)\times h_j,\,
u^\varepsilon
\bigr)_{H}
\bigr|
&\leq
\sum_{j=1}^{\infty}
\|h_j\|_{L^\infty}
\|G_j(u^\varepsilon)\|_{H}
\|u^\varepsilon\|_{H}
\\
&\leq
\sum_{j=1}^{\infty}
\|h_j\|_{L^\infty}
\Bigl(
\|h_j\|_{L^\infty}\|u^\varepsilon\|_{H}
+
\|h_j\|_{H}
+
\|\Delta h_j\|_{H}
\Bigr)
\|u^\varepsilon\|_{H}
\\
&\leq
C_h\bigl(1+\|u^\varepsilon\|_{H}^2\bigr).
\end{aligned}
\end{equation}

Similarly,
\begin{equation}\label{eq:gj-l2}
\begin{aligned}
\sum_{j=1}^{\infty}
\|G_j(u^\varepsilon)\|_{H}^2
&\leq
\sum_{j=1}^{\infty}
\Bigl(
\|u^\varepsilon\times h_j\|_{H}^2
+
\|h_j\|_{H^2}^2
\Bigr)
\\
&\leq
C_h\bigl(1+\|u^\varepsilon\|_{H}^2\bigr).
\end{aligned}
\end{equation}
Let us now define the following energy functional:
\begin{equation}\label{eq:energy-functional}
\begin{aligned}
|u^\varepsilon|_{L^4}^{H^2,H}(T)
:={}&
\sup_{0\le t\le T}
\|u^\varepsilon(t)\|_{H}^{2}
+
2\varepsilon\beta_{1}
\int_{0}^{T}
\|\nabla u^\varepsilon(s)\|_{H}^{2}\,ds
\\
&\quad
+
2\varepsilon\beta_{2}
\int_{0}^{T}
\|\Delta u^\varepsilon(s)\|_{H}^{2}\,ds
+
2\varepsilon\beta_{3}
\int_{0}^{T}
\|u^\varepsilon(s)\|_{L^{4}}^{4}\,ds
\\
&\quad
+
4\varepsilon\beta_{5}
\int_{0}^{T}
\|u^\varepsilon(s)\cdot\nabla u^\varepsilon(s)\|_{H}^{2}\,ds
\\
&\quad
+
2\varepsilon\beta_{5}
\int_{0}^{T}
\bigl\|
|u^\varepsilon(s)|\,|\nabla u^\varepsilon(s)|
\bigr\|_{H}^{2}\,ds .
\end{aligned}
\end{equation}
Therefore
\begin{equation}\label{eq:energy-estimate}
\begin{aligned}
|u^\varepsilon|_{L^4}^{H^2,H}(T)
\le\;&
\|u_0\|_{H}^{2}
+4\varepsilon C_h T+2\varepsilon(\beta_3+C_h)\int_0^T
\|u^\varepsilon(s)\|_{H}^2\,ds \\
&+
2\sqrt{\varepsilon}
\sup_{0\le t\le T}
\Biggl|
\sum_{j=1}^{\infty}
\int_0^t
\Bigl\langle
G_j(u^\varepsilon(s)),
u^\varepsilon(s)
\Bigr\rangle_H
\,dW_j(s)
\Biggr|.
\end{aligned}
\end{equation}
Then, for any $p\geq 2$, we obtain
\begin{equation}\label{eq:lp-estimate}
\begin{aligned}
\Bigl[
\mathbb{E}
\bigl(
|u^\varepsilon|_{L^4}^{H^2,H}(T)
\bigr)^p
\Bigr]^{\frac1p}
\le\;&
\|u_0\|_{H}^{2}
+4\varepsilon C_h T +
2\varepsilon(\beta_3+C_h)
\Biggl[
\mathbb{E}
\Biggl(
\int_0^T
|u^\varepsilon|_{L^4}^{H^2,H}(s)\,ds
\Biggr)^p
\Biggr]^{\frac1p}\\
&
+2\sqrt{\varepsilon}
\Biggl[\mathbb{E}\Biggl(\sup_{0\le t\le T}\Biggl|
\sum_{j=1}^{\infty}\int_0^t\Bigl\langle
G_j(u^\varepsilon(s)), u^\varepsilon(s)\Bigr\rangle_H\,dW_j(s)\Biggr|\Biggr)^p
\Biggr]^{\frac1p}.
\end{aligned}
\end{equation}
Minkowski's inequality yields
\[
\Biggl[
\mathbb{E}
\Biggl(
\int_0^T
|u^\varepsilon|_{L^4}^{H^2,H}(s)\,ds
\Biggr)^p
\Biggr]^{\frac1p}\le\int_0^T\Bigl(\mathbb{E}\bigl(|u^\varepsilon|_{L^4}^{H^2,H}(s)\bigr)^p\Bigr)^{\frac1p}\,ds.
\]

By the Burkholder--Davis--Gundy inequality and H\"older's inequality, we deduce
\begin{equation}\label{eq:bdg-estimate}
\begin{aligned}
&
2\sqrt{\varepsilon}\Biggl[\mathbb{E}\Biggl(\sup_{0\le t\le T}\Biggl|\sum_{j=1}^{\infty}\int_0^t\Bigl\langle G_j(u^\varepsilon(s)), u^\varepsilon(s)\Bigr\rangle_H\,dW_j(s)\Biggr|\Biggr)^p\Biggr]^{\frac1p} \leq
\sqrt{p\varepsilon}\,C_h\Biggl[\int_0^T\Bigl(1+\bigl(\mathbb{E}\|u^\varepsilon(s)\|_{H}^{2p}\bigr)^{2/p}\Bigr)\,ds
\Biggr]^{\frac12}.
\end{aligned}
\end{equation}
Substituting the above estimates into \eqref{eq:lp-estimate}, we get
\[
\begin{aligned}
\Biggl[
\mathbb{E}
\bigl(
|u^\varepsilon|_{L^4}^{H^2,H}(T)
\bigr)^p
\Biggr]^{\frac{2}{p}}
\le\;&
9\|u_0\|_{H}^4
+
12\varepsilon^2
C_{\beta_2,h}
\int_0^T
\Biggl[
\mathbb{E}
\bigl(
|u^\varepsilon|_{L^4}^{H^2,H}(s)
\bigr)^p
\Biggr]^{\frac{2}{p}}
\,ds
\\
&+
144C_{T,h}\varepsilon^2+9C_{T,h}p\varepsilon+3C_h^2p\varepsilon\int_0^T
\Biggl[\mathbb E\Bigl(|u^\varepsilon|_{L^4}^{H^2,H}(s)\Bigr)^p
\Biggr]^{\frac{2}{p}}
\,ds .
\end{aligned}
\]
Applying Gronwall's inequality, we obtain
\begin{equation}\label{eq:gronwall-estimate}
\begin{aligned}
\Biggl[
\mathbb{E}
\bigl(
|u^\varepsilon|_{L^4}^{H^2,H}(T)
\bigr)^p
\Biggr]^{\frac{2}{p}}
\le\;&
\Bigl(
9\|u_0\|_{H}^4
+
144C_{T,h}\varepsilon^2
+
9C_{T,h}p\varepsilon
\Bigr)
\exp\Bigl\{
12\varepsilon^2
C_{\beta_2,h}T
+
3C_h^2p\varepsilon T
\Bigr\}.
\end{aligned}
\end{equation}
Choosing \(p=\frac1\varepsilon\) in \eqref{eq:gronwall-estimate} and using Chebyshev's inequality, we obtain
\[
\begin{aligned}
\sup_{0<\varepsilon\le1}
\varepsilon
\log
\mathbb P\Bigl(
|u^\varepsilon|_{L^4}^{H^2,H}(T)>M
\Bigr)
\le\;&
-\log M
+\frac12\log\Bigl(
9\|u_0\|_{H}^4+36C_{T,h}\varepsilon^2+9C_{T,h}
\Bigr)
\\
&+
6\varepsilon^2 C_{\beta_2,T,h}T
+\frac32 C_h^2T.
\end{aligned}
\]

Therefore,
\begin{equation}{\label{eq:f1}}
\sup_{0<\varepsilon\le1}
\varepsilon
\log
\mathbb P\Bigl(
|u^\varepsilon|_{L^4}^{H^2,H}(T)>M
\Bigr)\to-\infty,
\qquad\text{as }M\to\infty .
\end{equation}

\medskip

We now apply It\^o's formula to
\(\frac12\|\nabla u^\varepsilon(t)\|_H^2\), we obtain
\begin{equation}\label{eq:h1-ito}
\begin{aligned}
\frac12\|\nabla u^\varepsilon(t)\|_{H}^2
&+\varepsilon\beta_1\int_0^t \|\Delta u^\varepsilon(s)\|_{H}^2\,ds
+\varepsilon\beta_2\int_0^t \|\nabla\Delta u^\varepsilon(s)\|_{H}^2\,ds +\varepsilon\beta_3\int_0^t
\Bigl\langle
\nabla\bigl(|u^\varepsilon(s)|^2u^\varepsilon(s)\bigr),
\nabla u^\varepsilon(s)
\Bigr\rangle_H\,ds \\
&\qquad=
\frac12\|\nabla u^\varepsilon(0)\|_{H}^2
+\varepsilon\beta_3\int_0^t
\|\nabla u^\varepsilon(s)\|_{H}^2\,ds +\varepsilon\beta_5\int_0^t
\Bigl\langle
\nabla\Delta\bigl(|u^\varepsilon(s)|^2u^\varepsilon(s)\bigr),
\nabla u^\varepsilon(s)
\Bigr\rangle_H\,ds \\
&-\frac{\varepsilon}{2}
\sum_{j=1}^{\infty}
\int_0^t
\Bigl\langle
\nabla\bigl(G_j(u^\varepsilon(s))\times h_j\bigr),
\nabla u^\varepsilon(s)
\Bigr\rangle_H\,ds +\frac{\varepsilon}{2}
\sum_{j=1}^{\infty}
\int_0^t
\|\nabla G_j(u^\varepsilon(s))\|_{H}^2\,ds \\
&+\sqrt{\varepsilon}
\sum_{j=1}^{\infty}
\int_0^t
\Bigl\langle
\nabla G_j(u^\varepsilon(s)),
\nabla u^\varepsilon(s)
\Bigr\rangle_H\,dW_j(s).
\end{aligned}
\end{equation}
By \eqref{eq:aux1}, we obtain
\begin{equation}\label{eq:h1-nonlinear1}
\begin{aligned}
\varepsilon\beta_3\int_0^t
\Bigl\langle
\nabla\bigl(|u^\varepsilon|^2u^\varepsilon\bigr),
\nabla u^\varepsilon
\Bigr\rangle_H \, ds
=
&\;
2\varepsilon\beta_3\int_0^t
\|u^\varepsilon\cdot\nabla u^\varepsilon\|_{H}^2 \, ds
+
\varepsilon\beta_3\int_0^t
\||u^\varepsilon|\,|\nabla u^\varepsilon|\|_{H}^2 \, ds .
\end{aligned}
\end{equation}

Moreover, by \eqref{eq:aux3} , we get
\begin{equation}\label{eq:h1-nonlinear2}
\begin{aligned}
\Bigl\langle
\nabla\Delta\bigl(|u|^2u\bigr),
\nabla u
\Bigr\rangle_H
&=
-
\Bigl\langle
\Delta\bigl(|u|^2u\bigr),
\Delta u
\Bigr\rangle_H
\\
&=
-2\Bigl\langle
|\nabla u|^2u,
\Delta u
\Bigr\rangle_H
-2\|u\cdot\Delta u\|_{H}^2
\\
&
-\||u|\,|\Delta u|\|_{H}^2
-4\Bigl\langle
\nabla u (u\cdot\nabla u)^\top,
\Delta u
\Bigr\rangle_H .
\end{aligned}
\end{equation}
Applying Young's inequality, we further infer that
\begin{equation}\label{eq:h1-young}
\begin{aligned}
2\Bigl|
\Bigl\langle
|\nabla u|^2u,
\Delta u
\Bigr\rangle_H
\Bigr|
&\le
\frac{1}{2}\||u|\,|\Delta u|\|_{H}^2
+ C_1\|\nabla u\|_{L^4}^4,
\\
4\Bigl|
\Bigl\langle
\nabla u (u\cdot\nabla u)^\top,
\Delta u
\Bigr\rangle_H
\Bigr|
&\le
\frac{1}{2}\||u|\,|\Delta u|\|_{H}^2
+ C_2\|\nabla u\|_{L^4}^4,
\end{aligned}
\end{equation}
where \(C_1,C_2>0\) are positive constants. Consequently, combining \eqref{eq:h1-nonlinear2} and \eqref{eq:h1-young}, we obtain
\begin{equation}\label{eq:h1-nonlinear3}
\begin{aligned}
\Bigl\langle
\Delta\bigl(|u|^2u\bigr),
\Delta u
\Bigr\rangle_H
&\le
-2\|u\cdot\Delta u\|_{H}^2
+
(C_1+C_2)\|\nabla u\|_{L^4}^4 .
\end{aligned}
\end{equation}
Therefore,
\begin{equation}\label{eq:h1-nonlinear4}
\begin{aligned}
-\varepsilon\beta_5\int_0^t
\Bigl\langle
\nabla\Delta\bigl(|u^\varepsilon|^2u^\varepsilon\bigr),
\nabla u^\varepsilon
\Bigr\rangle_H ds
&=
\varepsilon\beta_5\int_0^t
\Bigl\langle
\Delta\bigl(|u^\varepsilon|^2u^\varepsilon\bigr),
\Delta u^\varepsilon
\Bigr\rangle_H ds
\\
&\le
\varepsilon\int_0^t
\Bigl(
-2\beta_5\|u^\varepsilon\cdot\Delta u^\varepsilon\|_{H}^2
+\beta_5 C_3\|\nabla u^\varepsilon\|_{L^4}^4
\Bigr)\,ds ,
\end{aligned}
\end{equation}
where \(C_3=C_1+C_2\).

Next, we estimate the term \(\|\nabla u^\varepsilon\|_{L^4}^4\) using the Gagliardo--Nirenberg inequality. In the one-dimensional case, we have
\begin{equation}\label{eq:GN}
\|\nabla u^\varepsilon\|_{L^4}^4
\leq
\|u^\varepsilon\|_{H}^{\frac72}
\|u^\varepsilon\|_{H^3}^{\frac12}
\le
\frac{\beta_2}{2C_3\beta_5}
\|u^\varepsilon\|_{H^3}^2
+
C_{\beta_2,\beta_5}
\|u^\varepsilon\|_{H}^{14}.
\end{equation}
Moreover, by the standard elliptic regularity estimate with homogeneous Neumann boundary conditions\cite{gri}, it follows that
\begin{equation}\label{eq:elliptic}
\|u^\varepsilon\|_{H^3}^2
\eqsim
\|u^\varepsilon\|_{H}^2
+
\|\nabla u^\varepsilon\|_{H}^2
+
\|\nabla\Delta u^\varepsilon\|_{H}^2 .
\end{equation}

Combining \eqref{eq:GN} and \eqref{eq:elliptic}, we deduce that
\begin{equation}\label{eq:GN-final}
\begin{aligned}
\|\nabla u^\varepsilon\|_{L^4}^4
\le\;&
\frac{\beta_2}{2C_3\beta_5}
\|\nabla\Delta u^\varepsilon\|_{H}^2
+
\frac{\beta_2}{2C_3\beta_5}
\|\nabla u^\varepsilon\|_{H}^2
+
C_{\beta_2,\beta_5}
\bigl(
1+\|u^\varepsilon\|_{H}^{14}
\bigr).
\end{aligned}
\end{equation}
Using assumption \(
\sum_{j=1}^{\infty}\|h_j\|_{H^3}<C_h<\infty,
\) together with H\"older's and Young's inequalities, we obtain
\begin{equation}\label{eq:h1-gj1}
\begin{aligned}
\sum_{j=1}^{\infty}
\Bigl|
\Bigl\langle
\nabla\bigl(G_j(u^\varepsilon)\times h_j\bigr),
\nabla u^\varepsilon
\Bigr\rangle_H
\Bigr|
&=
\sum_{j=1}^{\infty}
\Bigl|
\Bigl\langle
G_j(u^\varepsilon)\times h_j,
\Delta u^\varepsilon
\Bigr\rangle_H
\Bigr|
\\
&\le
\sum_{j=1}^{\infty}
\|h_j\|_{L^\infty}
\|G_j(u^\varepsilon)\|_{H}
\|\Delta u^\varepsilon\|_{H}
\\
&\le
\sum_{j=1}^{\infty}
\|h_j\|_{L^\infty}
\Bigl(
\|h_j\|_{L^\infty}\|u^\varepsilon\|_{H}
+\|h_j\|_{H}
+\|\Delta h_j\|_{H}
\Bigr)
\|\Delta u^\varepsilon\|_{H}
\\
&\le
{\beta_1}\|\Delta u^\varepsilon\|_{H}^2
+
C_{h,\beta_1}\|u^\varepsilon\|_{H}^2
+
C_{h,\beta_1}.
\end{aligned}
\end{equation}

Therefore,
\begin{equation}\label{eq:h1-gj2}
\begin{aligned}
-\frac{\varepsilon}{2}
\sum_{j=1}^{\infty}
\Bigl\langle
\nabla\bigl(G_j(u^\varepsilon)\times h_j\bigr),
\nabla u^\varepsilon
\Bigr\rangle_H
\le\;&
\frac{\varepsilon\beta_1}{2}
\|\Delta u^\varepsilon\|_{H}^2
+
\frac{\varepsilon}{2}
\Bigl(
C_{h,\beta_1}\|u^\varepsilon\|_{H}^2
+
C_{h,\beta_1}
\Bigr).
\end{aligned}
\end{equation}
Also we have,
\begin{equation}\label{eq:h1-gj3}
\begin{aligned}
\sum_{j=1}^{\infty}
\|\nabla G_j(u^\varepsilon)\|_{H}^2
&\leq
\sum_{j=1}^{\infty}
\Bigl(
\|\nabla(u^\varepsilon\times h_j)\|_{H}^2
+
\|\nabla h_j\|_{H}^2
+
\|\nabla\Delta h_j\|_{H}^2
\Bigr)
\\
&\le
C_h\bigl(1+\|\nabla u^\varepsilon\|_{H}^2\bigr).
\end{aligned}
\end{equation}

Hence,
\begin{equation}\label{eq:h1-gj4}
\frac{\varepsilon}{2}
\sum_{j=1}^{\infty}
\|\nabla G_j(u^\varepsilon)\|_{H}^2
\le
\frac{\varepsilon}{2}
C_h
\bigl(
1+\|\nabla u^\varepsilon\|_{H}^2
\bigr).
\end{equation}

Using \eqref{eq:GN-final}, \eqref{eq:h1-gj2}, and \eqref{eq:h1-gj4}, and substituting these estimates into \eqref{eq:h1-ito}, together with \eqref{eq:h1-nonlinear1} and \eqref{eq:h1-nonlinear4}, we obtain
\begin{equation}\label{eq:final_energy_estimate2}
\begin{aligned}
\|\nabla u^\varepsilon(t)\|_{H}^2
&+\varepsilon\beta_1\int_0^t \|\Delta u^\varepsilon(s)\|_{H}^2\,ds
+\varepsilon\beta_2\int_0^t \|\nabla\Delta u^\varepsilon(s)\|_{H}^2\,ds \\
&+2\varepsilon\beta_3\int_0^t
\Bigl(
2\|u^\varepsilon(s)\cdot\nabla u^\varepsilon(s)\|_{H}^2
+\bigl\||u^\varepsilon(s)|\,|\nabla u^\varepsilon(s)|\bigr\|_{H}^2
\Bigr)\,ds +4\varepsilon\beta_5\int_0^t
\|u^\varepsilon(s)\cdot\Delta u^\varepsilon(s)\|_{H}^2\,ds \\
&\le
\|\nabla u^\varepsilon(0)\|_{H}^2
+2\varepsilon C_{\beta_2,\beta_5}
\int_0^t \|u^\varepsilon(s)\|_{H}^{14}\,ds +\varepsilon C_{h,\beta_1}\int_0^t \|u^\varepsilon(s)\|_{H}^2\,ds\\
&+\varepsilon
\bigl(2C_h+\beta_2+2\beta_3\bigr)
\int_0^t \|\nabla u^\varepsilon(s)\|_{H}^2\,ds +\varepsilon
\bigl(C_{h,\beta_1}+C_h+2C_{\beta_2,\beta_5}\bigr)t \\
&+2\sqrt{\varepsilon}
\sum_{j=1}^{\infty}
\int_0^t
\Bigl\langle
\nabla G_j(u^\varepsilon(s)),
\nabla u^\varepsilon(s)
\Bigr\rangle_H\,dW_j(s).
\end{aligned}
\end{equation}
Let us now define the following functional
\begin{equation}
\begin{aligned}
\|u^\varepsilon\|_{H^1,H^2}^{H^3}(t)
:=\;&
\|\nabla u^\varepsilon(t)\|_{H}^2
+\varepsilon\beta_1\int_0^t \|\Delta u^\varepsilon(s)\|_{H}^2\,ds +
\varepsilon\beta_2\int_0^t \|\nabla\Delta u^\varepsilon(s)\|_{H}^2\,ds \\
&+
4\varepsilon\beta_3\int_0^t
\|u^\varepsilon(s)\cdot\nabla u^\varepsilon(s)\|_{H}^2\,ds +
2\varepsilon\beta_3\int_0^t
\bigl\||u^\varepsilon(s)|\,|\nabla u^\varepsilon(s)|\bigr\|_{H}^2\,ds \\
&+
4\varepsilon\beta_5\int_0^t
\|u^\varepsilon(s)\cdot\Delta u^\varepsilon(s)\|_{H}^2\,ds .
\end{aligned}
\end{equation}
Then, for any \(p\ge2\), by using \eqref{eq:gronwall-estimate} and Minkowski's inequality, we obtain
\begin{equation}\label{eq:higher_moment_estimate}
\begin{aligned}
\Bigl[
\mathbb{E}
\Bigl(
\|u^\varepsilon\|_{H^1,H^2}^{H^3}(T)
\Bigr)^p
\Bigr]^{1/p}&
\le
\|\nabla u^\varepsilon(0)\|_{H}^2
+
\varepsilon C_{\beta_2,\beta_5,\varepsilon,h,T}+\varepsilon\bigl(2C_h+\beta_2+2\beta_3\bigr)\int_0^T\Bigl[\mathbb{E}
\|\nabla u^\varepsilon(s)\|_{H}^{2p}
\Bigr]^{1/p}
\,ds
\\
&\quad
+
\varepsilon C_{h,\beta_1,\beta_2,\beta_5}T+2\sqrt{\varepsilon}\Biggl[\mathbb{E}\Biggl(\sup_{0\le T}\Biggl|\sum_{j=1}^{\infty}\int_0^t
\Bigl\langle
\nabla G_j(u^\varepsilon(s)),
\nabla u^\varepsilon(s)
\Bigr\rangle_H
\,dW_j(s)
\Biggr|
\Biggr)^p
\Biggr]^{1/p}.
\end{aligned}
\end{equation}
On the other hand, by the Burkholder--Davis--Gundy inequality and
H\"older's inequality, we deduce
\begin{equation}\label{eq:bdg_higher}
\begin{aligned}
&
2\sqrt{\varepsilon}
\Biggl[
\mathbb{E}
\Biggl(
\sup_{0\le t\le T}
\Biggl|
\sum_{j=1}^{\infty}
\int_0^t
\Bigl\langle
\nabla G_j(u^\varepsilon(s)),
\nabla u^\varepsilon(s)
\Bigr\rangle_H
\,dW_j(s)
\Biggr|
\Biggr)^p
\Biggr]^{1/p}
\\
&\leq
\sqrt{p\varepsilon}\,C_h
\Biggl[
\int_0^T
\Bigl(
1+
\bigl(
\mathbb{E}
\|\nabla u^\varepsilon(s)\|_H^{2p}
\bigr)^{2/p}
\Bigr)
\,ds
\Biggr]^{1/2}.
\end{aligned}
\end{equation}
Substituting \eqref{eq:bdg_higher} into
\eqref{eq:higher_moment_estimate}, we obtain
\begin{equation}
\begin{aligned}
\Biggl[
\mathbb{E}
\Bigl(
\|u^\varepsilon\|_{H^1,H^2}^{H^3}(T)
\Bigr)^p
\Biggr]^{2/p}\le&
9\|\nabla u^\varepsilon(0)\|_{H}^{4}
+9 C_{T,h}p\varepsilon+ C_{h,\beta_1,\beta_2,\beta_3,\beta_5,T}\varepsilon^{2}\\
&\qquad
+ \bigl( \varepsilon^2 C_{h,\beta_2 ,\beta_3} + C_{h}^{2}p\varepsilon \bigr)
\int_{0}^{T}
\Biggl[
\mathbb{E}
\Bigl(
\|u^\varepsilon\|_{H^1,H^2}^{H^3}(s)
\Bigr)^p
\Biggr]^{2/p}
\,ds .
\end{aligned}
\end{equation}
Applying Gronwall's inequality, we conclude that
\begin{equation}
\begin{aligned}
&
\Biggl[
\mathbb{E}
\Bigl(
\|u^\varepsilon\|_{H^1,H^2}^{H^3}(T)
\Bigr)^p
\Biggr]^{2/p}
\\
&\le
\Bigl(
9\|\nabla u^\varepsilon(0)\|_{H}^4
+
C_{T,h}p\varepsilon
+
C_{h,\beta_1,\beta_2,\beta_3,\beta_5,T}\varepsilon^{2}
\Bigr)
\exp\Bigl\{
\bigl( \varepsilon^2 C_{h,\beta_2,\beta_3} + C_{h}^{2}p\varepsilon \bigr) T
\Bigr\}.
\end{aligned}
\end{equation}
Choosing \(p = \frac{1}{\varepsilon}\) and simplifying, we have
\[
\begin{aligned}
\sup_{0<\varepsilon\le 1}
\varepsilon
\log
\mathbb P\Bigl(
\|u^\varepsilon\|_{H^1,H^2}^{H^3}(T) > M
\Bigr)
\le\;&
-\log M
+ \frac12 \log\Bigl(
9\|\nabla u^\varepsilon(0)\|_{H}^4
+ C_{h,\beta_1,\beta_2,\beta_3,\beta_5,T}\varepsilon^{2}
+ C_{T,h}
\Bigr)
\\
&
+ \frac{C_{h,\beta_2,\beta_5}\varepsilon^2 T}{2}
+ \frac{C_h^2 T}{2}.
\end{aligned}
\]

Therefore,
\begin{equation}{\label{eq:f2}}
\sup_{0<\varepsilon\le 1}
\varepsilon
\log
\mathbb P\Bigl(
\|u^\varepsilon\|_{H^1,H^2}^{H^3}(T) > M
\Bigr) \to -\infty,
\qquad \text{as } M \to \infty .
\end{equation}
From \eqref{eq:f1} and \eqref{eq:f2} we can easily conclude
\[
\lim_{M\to\infty}
\sup_{0<\varepsilon\le1}
\varepsilon
\log
\mathbb{P}
\left(
\|u^\varepsilon\|_{H^1}^{H^2}(T) > M
\right)
=
-\infty.
\]
This completes the proof.
\end{proof}
Since \(H^2\) is dense in \(H^1\), there exists a sequence
\(
(x_n)_{n\in\mathbb N}\subset H^2
\)
such that
\[
\lim_{n\to\infty}
\|x_n-x\|_{H^1}=0.
\]

Let \(u_n^\varepsilon\) be the solution of Eq.~\eqref{eq:ito_LLBar}
with initial value \(x_n\).
Similar to Lemma~\ref{lemma1}, one obtains
\[
\lim_{M\to\infty}\sup_{n\in\mathbb N}
\sup_{0<\varepsilon\le1}
\varepsilon
\log
\mathbb P
\Bigl(
\|u_n^\varepsilon\|_{H^1}^{H^2}(T)>M
\Bigr)
=-\infty,
\]
where
\[
\|u_n^\varepsilon\|_{H^1}^{H^2}(T)
:=
\sup_{0\le t\le T}
\|u_n^\varepsilon(t)\|_{H^1}^{2}
+
2C\varepsilon
\int_{0}^{T}
\|u_n^\varepsilon(s)\|_{H^2}^{2}\,ds .
\]
\begin{lemma}{\label{lemma2}}
For any $\delta>0$,
\[
\lim_{n\to\infty}\sup_{0<\varepsilon\le1}
\varepsilon
\log
\mathbb P
\left(
\sup_{0\le t\le T}
\|u^\varepsilon(t)-u_n^\varepsilon(t)\|_{H}^2>\delta
\right)
=-\infty .
\]
\end{lemma}
\begin{proof}
Let \( w_n^\varepsilon = u_n^\varepsilon - u^\varepsilon \). Then \( w_n^\varepsilon \) satisfies
\[
\begin{aligned}
w_n^\varepsilon(t) &= (x_n - x) + \varepsilon \beta_1 \int_0^t \Delta w_n^\varepsilon(s) \,ds - \varepsilon \beta_2 \int_0^t \Delta^2 w_n^\varepsilon(s) \,ds + \varepsilon \beta_3 \int_0^t \bigl[ (1 - |u_n^\varepsilon|^2) u_n^\varepsilon - (1 - |u^\varepsilon|^2) u^\varepsilon \bigr] \,ds   \\
&\quad- \varepsilon \beta_4 \int_0^t \bigl[ u_n^\varepsilon \times \Delta u_n^\varepsilon - u^\varepsilon \times \Delta u^\varepsilon \bigr] \,ds + \varepsilon \beta_5 \int_0^t \Delta \bigl( |u_n^\varepsilon|^2 u_n^\varepsilon - |u^\varepsilon|^2 u^\varepsilon \bigr) \,ds
+ \frac{\varepsilon}{2} \sum_{j=1}^\infty \int_0^t (w_n^\varepsilon \times h_j) \times h_j \,ds  \\
&\quad\qquad\qquad- \sqrt{\varepsilon} \sum_{j=1}^\infty \int_0^t (w_n^\varepsilon \times h_j) \,dW_j(s).
\end{aligned}
\]
Applying the Itô formula to $\|w_n^\varepsilon(t)\|_H^2$, we obtain
\begin{equation}
\begin{aligned}
\|w_n^\varepsilon(t)\|_{H}^{2}
&= \|x_n-x\|_{H}^{2}
+ 2\varepsilon\beta_{1}\int_{0}^{t} \langle w_n^\varepsilon,\Delta w_n^\varepsilon\rangle_{H}\,ds
- 2\varepsilon\beta_{2}\int_{0}^{t} \langle w_n^\varepsilon,\Delta^{2}w_n^\varepsilon\rangle_{H}\,ds \\
&\quad + 2\varepsilon\beta_{3}\int_{0}^{t} \bigl\langle w_n^\varepsilon,\,
(1-|u_n^\varepsilon|^{2})u_n^\varepsilon - (1-|u^\varepsilon|^{2})u^\varepsilon \bigr\rangle_{H}\,ds 
- 2\varepsilon\beta_{4}\int_{0}^{t} \bigl\langle w_n^\varepsilon,\,
u_n^\varepsilon\times\Delta u_n^\varepsilon - u^\varepsilon\times\Delta u^\varepsilon \bigr\rangle_{H}\,ds \\
&\quad + 2\varepsilon\beta_{5}\int_{0}^{t} \bigl\langle w_n^\varepsilon,\,
\Delta\bigl( |u_n^\varepsilon|^{2}u_n^\varepsilon - |u^\varepsilon|^{2}u^\varepsilon \bigr) \bigr\rangle_{H}\,ds 
+ \varepsilon\sum_{j=1}^{\infty}\int_{0}^{t} \bigl\langle w_n^\varepsilon,\,
(w_n^\varepsilon\times h_j)\times h_j \bigr\rangle_{H}\,ds \\
&\quad - 2\sqrt{\varepsilon}\sum_{j=1}^{\infty}\int_{0}^{t} \bigl\langle w_n^\varepsilon,\,
w_n^\varepsilon\times h_j \bigr\rangle_{H}\,dW_j(s) 
+ \varepsilon\sum_{j=1}^{\infty}\int_{0}^{t} \|w_n^\varepsilon\times h_j\|_{H}^{2}\,ds .
\end{aligned}
\end{equation}
Using the facts that
\(
\bigl\langle w_n^\varepsilon\times h_j,\; w_n^\varepsilon \bigr\rangle_H = 0,
\)
and
\(
\bigl\langle (w_n^\varepsilon\times h_j)\times h_j,\;
w_n^\varepsilon \bigr\rangle_H
=
-\,\|w_n^\varepsilon\times h_j\|_H^2,
\)
we obtain
\[
\begin{aligned}
&
2\sqrt{\varepsilon}
\sum_{j=1}^{\infty}
\int_0^t
\bigl\langle
w_n^\varepsilon\times h_j,\;
w_n^\varepsilon
\bigr\rangle_H
\,dW_j(s)
\\
&\qquad
+
\varepsilon
\sum_{j=1}^{\infty}
\int_0^t
\bigl\langle(w_n^\varepsilon\times h_j)\times h_j,\;w_n^\varepsilon\bigr\rangle_H\,ds
\\
&=-\varepsilon\sum_{j=1}^{\infty}\int_0^t\|w_n^\varepsilon\times h_j\|_H^2\,ds .
\end{aligned}
\]
Hence, we obtain
\begin{equation}{\label{wn_H}}
\begin{aligned}
\|w_n^\varepsilon(t)\|_H^2
&+ 2\varepsilon\beta_1 \int_0^t \|\nabla w_n^\varepsilon(s)\|_H^2 \,ds
+ 2\varepsilon\beta_2 \int_0^t \|\Delta w_n^\varepsilon(s)\|_H^2 \,ds \\
&= \|x_n-x\|_H^2
+ 2\varepsilon\beta_3 \int_0^t
\Bigl\langle (1-|u_n^\varepsilon|^2)u_n^\varepsilon - (1-|u^\varepsilon|^2)u^\varepsilon,\;
w_n^\varepsilon \Bigr\rangle_H \,ds \\
&\quad - 2\varepsilon\beta_4 \int_0^t
\Bigl\langle u_n^\varepsilon\times\Delta u_n^\varepsilon - u^\varepsilon\times\Delta u^\varepsilon,\;
w_n^\varepsilon \Bigr\rangle_H \,ds \\
&\quad + 2\varepsilon\beta_5 \int_0^t
\Bigl\langle \Delta\bigl( |u_n^\varepsilon|^2u_n^\varepsilon - |u^\varepsilon|^2u^\varepsilon \bigr),\;
w_n^\varepsilon \Bigr\rangle_H \,ds .
\end{aligned}
\end{equation}

For convenience, we define the terms as follows:
\[
\begin{aligned}
I_1(t) &:= 2\varepsilon\beta_3 \int_0^t
\Bigl\langle (1-|u_n^\varepsilon|^2)u_n^\varepsilon - (1-|u^\varepsilon|^2)u^\varepsilon,\;
w_n^\varepsilon \Bigr\rangle_H \,ds, \\[4pt]
I_2(t) &:= -2\varepsilon\beta_4 \int_0^t
\Bigl\langle u_n^\varepsilon\times\Delta u_n^\varepsilon - u^\varepsilon\times\Delta u^\varepsilon,\;
w_n^\varepsilon \Bigr\rangle_H \,ds, \\[4pt]
I_3(t) &:= 2\varepsilon\beta_5 \int_0^t
\Bigl\langle \Delta\bigl( |u_n^\varepsilon|^2u_n^\varepsilon - |u^\varepsilon|^2u^\varepsilon \bigr),\;
w_n^\varepsilon \Bigr\rangle_H \,ds .
\end{aligned}
\]
Thus the equation \eqref{wn_H} can be written as
\[
\|w_n^\varepsilon(t)\|_H^2
+ 2\varepsilon\beta_1 \int_0^t \|\nabla w_n^\varepsilon(s)\|_H^2 \,ds
+ 2\varepsilon\beta_2 \int_0^t \|\Delta w_n^\varepsilon(s)\|_H^2 \,ds
= \|x_n-x\|_H^2 + I_1(t) + I_2(t) + I_3(t).
\]
\(I_1\) can also be written as
\[
\begin{aligned}
I_1(t) &= 2\varepsilon\beta_3 \int_0^t \|w_n^\varepsilon\|_H^2 \,ds
- 2\varepsilon\beta_3 \int_0^t
\Bigl\langle |u_n^\varepsilon|^2 u_n^\varepsilon - |u^\varepsilon|^2 u^\varepsilon,\;
w_n^\varepsilon \Bigr\rangle_H \,ds .
\end{aligned}
\]
Using the identity
\(
|u_n^\varepsilon|^2 u_n^\varepsilon - |u^\varepsilon|^2 u^\varepsilon
= |u_n^\varepsilon|^2 w_n^\varepsilon + (|u_n^\varepsilon|^2 - |u^\varepsilon|^2) u^\varepsilon,
\)
we obtain
\[
\begin{aligned}
|I_1(t)| &\leq 2\varepsilon\beta_3 \int_0^t \|w_n^\varepsilon\|_H^2 \,ds
+ 2\varepsilon\beta_3 \int_0^t \bigl\| |u_n^\varepsilon|^2 w_n^\varepsilon \bigr\|_H \|w_n^\varepsilon\|_H \,ds \\
&\quad + 2\varepsilon\beta_3 \int_0^t \bigl\| (|u_n^\varepsilon|^2 - |u^\varepsilon|^2) u^\varepsilon \bigr\|_H \|w_n^\varepsilon\|_H \,ds .
\end{aligned}
\]
 Since \(H^1 \hookrightarrow L^\infty\) in 1D, we have the bounds
\[
\||u_n^\varepsilon|^2 w_n^\varepsilon\|_H \leq C \|u_n^\varepsilon\|_{H^1}^2 \|w_n^\varepsilon\|_H,
\]
and
\[
\bigl\| (|u_n^\varepsilon|^2 - |u^\varepsilon|^2) u^\varepsilon \bigr\|_H
\leq C \bigl( \|u_n^\varepsilon\|_{H^1} + \|u^\varepsilon\|_{H^1} \bigr) \|u^\varepsilon\|_{H^1} \|w_n^\varepsilon\|_H .
\]

Thus,
\[
\begin{aligned}
|I_1(t)| &\leq 2\varepsilon\beta_3 \int_0^t \|w_n^\varepsilon\|_H^2 \,ds \\
&\quad + C\varepsilon\beta_3 \int_0^t \|u_n^\varepsilon\|_{H^1}^2 \|w_n^\varepsilon\|_H^2 \,ds \\
&\quad + C\varepsilon\beta_3 \int_0^t \bigl( \|u_n^\varepsilon\|_{H^1} + \|u^\varepsilon\|_{H^1} \bigr) \|u^\varepsilon\|_{H^1} \|w_n^\varepsilon\|_H^2 \,ds .
\end{aligned}
\]
Similarly, using \( w_n^\varepsilon = u_n^\varepsilon - u^\varepsilon \), we can rewrite the term \( I_2 \) as
\[
\begin{aligned}
I_2(t) &= -2\varepsilon\beta_4 \int_0^t \langle u_n^\varepsilon \times \Delta u_n^\varepsilon - u^\varepsilon \times \Delta u^\varepsilon,\; w_n^\varepsilon \rangle_H \,ds \\
&= -2\varepsilon\beta_4 \int_0^t \langle u_n^\varepsilon \times \Delta w_n^\varepsilon,\; w_n^\varepsilon \rangle_H \,ds
- 2\varepsilon\beta_4 \int_0^t \langle w_n^\varepsilon \times \Delta u^\varepsilon,\; w_n^\varepsilon \rangle_H \,ds .
\end{aligned}
\]

The second term vanishes because
\[
\langle w_n^\varepsilon \times \Delta u^\varepsilon,\; w_n^\varepsilon \rangle_H = 0
\quad \text{since } (a \times b) \cdot a = 0 .
\]

Thus, by the embedding \( H^1 \hookrightarrow L^\infty \) in 1D, we have \( \|u_n^\varepsilon\|_{L^\infty} \leq C \|u_n^\varepsilon\|_{H^1} \). Then Young's inequality yields
\[
|I_2(t)| \leq \frac{\varepsilon\beta_2}{4} \int_0^t \|\Delta w_n^\varepsilon\|_H^2 \,ds
+ C\varepsilon \frac{\beta_4^2}{\beta_2} \int_0^t \|u_n^\varepsilon\|_{H^1}^2 \|w_n^\varepsilon\|_H^2 \,ds .
\]
Also, \( I_3 \) can be rewritten as
\[
I_3(t) := 2\varepsilon\beta_5 \int_0^t
\Bigl\langle |u_n^\varepsilon|^2 u_n^\varepsilon - |u^\varepsilon|^2 u^\varepsilon,\;
\Delta w_n^\varepsilon \Bigr\rangle_H \,ds .
\]
Similar to \(I_2\), by H\"older's inequality and Young's inequality, we obtain
\[
\begin{aligned}
|I_3(t)|
&\leq
2\varepsilon\beta_5
\int_0^t
\Bigl\|
|u_n^\varepsilon|^2u_n^\varepsilon
-
|u^\varepsilon|^2u^\varepsilon
\Bigr\|_{H}
\|\Delta w_n^\varepsilon\|_{H}
\,ds
\\[4pt]
&\leq
\frac{\varepsilon\beta_2}{4}
\int_0^t
\|\Delta w_n^\varepsilon\|_{H}^2
\,ds
\\
&\quad
+
C\varepsilon
\frac{\beta_5^2}{\beta_2}
\int_0^t
\bigl(
\|u_n^\varepsilon\|_{H^1}^4
+
\|u^\varepsilon\|_{H^1}^4
\bigr)
\|w_n^\varepsilon\|_{H}^2
\,ds .
\end{aligned}
\]
Hence,
\[
\begin{aligned}
\|w_n^\varepsilon(t)\|_{H}^2 &\leq \|x_n - x\|_{H}^2 + \int_0^t \Bigg[ 2\varepsilon\beta_3 
+ C\varepsilon\beta_3 \|u_n^\varepsilon\|_{H^1}^2 \\
&\qquad + C\varepsilon\beta_3 \bigl( \|u_n^\varepsilon\|_{H^1} + \|u^\varepsilon\|_{H^1} \bigr) \|u^\varepsilon\|_{H^1}  + C\varepsilon \frac{\beta_4^2}{\beta_2} \|u_n^\varepsilon\|_{H^1}^2 \\
&\qquad + C\varepsilon \frac{\beta_5^2}{\beta_2} \bigl( \|u_n^\varepsilon\|_{H^1}^4 + \|u^\varepsilon\|_{H^1}^4 \bigr) \Bigg] \|w_n^\varepsilon\|_H^2 \,ds .
\end{aligned}
\]
For $M > 0$, define stopping times
\[
\vartheta_{\varepsilon,M}^1 = \inf \left\{ t \ge 0 : \int_0^t  \|u^\varepsilon(s)\|_{H^2}^2\, ds  \text{ or }\|u^\varepsilon(s)\|_{H^1}^2  > M \right\}
\]
and
\[
\vartheta_{\varepsilon,M}^2 = \inf \left\{ t \ge 0 : \|u_n^\varepsilon(t)\|_{H^1} > M \right\}.
\]
Let $\vartheta_{\varepsilon,M} = \vartheta_{\varepsilon,M}^1 \wedge \vartheta_{\varepsilon,M}^2$. Then Gronwall's inequality implies
\[
\sup_{0 \le t \le \vartheta_{\varepsilon,M}} \|w_n^\varepsilon(t)\|_{H}^2 \le \|x_n - x\|_{H}^2 \, \exp\bigl( C_{\varepsilon}^M \bigr),
\]
where
\[
C_{\varepsilon}^M := 2\varepsilon\beta_3  + 3C\varepsilon\beta_3 M^2 + C\varepsilon\frac{\beta_4^2}{\beta_2}M + C\varepsilon\frac{\beta_5^2}{\beta_2} M^4 .
\]
Now, taking expectation and using the definition $w_n^\varepsilon = u_n^\varepsilon - u^\varepsilon$, we obtain
\[
\mathbb{E} \left( \sup_{0 \le s \le \vartheta_{\varepsilon,M}} \|u^\varepsilon(s) - u_n^\varepsilon(s)\|_{H}^2 \right)
\leq \|x - x_n\|_{H}^2 \, \exp\bigl( C_{\varepsilon}^M \bigr).
\]

Consequently,
\[
\left( \mathbb{E} \left( \sup_{0 \le s \le \vartheta_{\varepsilon,M}} \|u^\varepsilon(s) - u_n^\varepsilon(s)\|_{H}^{2p} \right) \right)^{2/p}
\leq  \exp\bigl\{ 2C_{\varepsilon}^M \bigr\} \, \|x - x_n\|_{H}^{4}.
\]

\medskip
Taking \( p = \frac{2}{\varepsilon} \), it follows that for any \( M > 0 \),

\[
\sup_{0<\varepsilon\leq 1} \varepsilon \log \mathbb{P} \left( \sup_{0 \le s \le T \wedge \vartheta_{\varepsilon,M}} \|u^\varepsilon(s) - u_n^\varepsilon(s)\|_{H}^2 > \delta \right)
\]
\[
\leq \sup_{0<\varepsilon\leq 1} \varepsilon \log \left( \frac{ \mathbb{E} \left[ \sup_{0 \le s \le T \wedge \vartheta_{\varepsilon,M}} \|u^\varepsilon(s) - u_n^\varepsilon(s)\|_{H}^{2p} \right] }{ \delta^p } \right)
\]
\[
\leq \sup_{0<\varepsilon\leq 1} \varepsilon \left( \log \left( \exp\{ 2C_M \varepsilon T \} \|x - x_n\|_{H}^{4} \right) - \frac{1}{\varepsilon} \log \delta \right)
\]
\[
\to -\infty \quad \text{as } n \to \infty.
\]
For any \(K>0\), by Lemma~\ref{lemma1}, there exists \(M>0\) such that, for all \(0<\varepsilon\le1\) and all \(n\ge1\),
\[
\mathbb P\Bigl(
\|u^\varepsilon\|_{H^1}^{H^2}(T)>M
\Bigr)
\le
e^{-K/\varepsilon},
\qquad
\mathbb P\Bigl(
\|u_n^\varepsilon\|_{H^1}^{H^2}(T)>M
\Bigr)
\le
e^{-K/\varepsilon}.
\]
For such $M$, there exists a positive integer $N$, such that for any $n \ge N$,
\[
\sup_{0<\varepsilon\le 1} \varepsilon \log \mathbb{P} \left( \sup_{0\le t\le T} \|u^\varepsilon(t) - u_n^\varepsilon(t)\|_{H}^2 > \delta, \;
\|u^\varepsilon\|_{H^1}^{H^2}(T) \le M, \; \|u_n^\varepsilon\|_{H^1}^{H^2}(T) \le M \right)
\]
\[
\le \sup_{0<\varepsilon\le 1} \varepsilon \log \mathbb{P} \left( \sup_{0\le t\le T \wedge \vartheta_{\varepsilon,K}} \|u^\varepsilon(t) - u_n^\varepsilon(t)\|_{H}^2 > \delta \right) \le -K.
\]
Combining the above estimates, we obtain
\[
\begin{aligned}
&
\mathbb P\left(
\sup_{0\le t\le T}
\|u^\varepsilon(t)-u_n^\varepsilon(t)\|_{H}^{2}
>\delta
\right)
\\
&\le
\mathbb P\Bigl(
\|u^\varepsilon\|_{H^1}^{H^2}(T)>M
\Bigr)
+
\mathbb P\Bigl(
\|u_n^\varepsilon\|_{H^1}^{H^2}(T)>M
\Bigr)
\\
&\quad
+
\mathbb P\Bigl(
\sup_{0\le t\le T}
\|u^\varepsilon(t)-u_n^\varepsilon(t)\|_{H}^{2}
>\delta,
\
\|u^\varepsilon\|_{H^1}^{H^2}(T)\le M,
\
\|u_n^\varepsilon\|_{H^1}^{H^2}(T)\le M
\Bigr)
\\
&\le
3e^{-K/\varepsilon}.
\end{aligned}
\]
Since \(K\) is arbitrary, the proof is complete.
\end{proof}
\begin{lemma}{\label{lemma3}}
For any $\delta>0$,
\[
\lim_{n\to\infty}\sup_{0<\varepsilon\le1}
\varepsilon
\log
\mathbb P
\left(
\sup_{0\le t\le T}
\|u_n^\varepsilon(t)-v_n^\varepsilon(t)\|_{H^1}^2>2\delta
\right)
=-\infty .
\]
\end{lemma}
\begin{proof}
Applying It\^o's formula
\(
\frac12\|\nabla w_n^\varepsilon(t)\|_{H}^2,
\)
we obtain
\begin{equation}
\begin{aligned}
\frac12\|\nabla w_n^\varepsilon(t)\|_{H}^2
&+
\varepsilon\beta_1
\int_0^t
\|\Delta w_n^\varepsilon(s)\|_{H}^2\,ds
+
\varepsilon\beta_2
\int_0^t
\|\nabla\Delta w_n^\varepsilon(s)\|_{H}^2\,ds
\\
&=
\frac12\|\nabla(x_n-x)\|_{H}^2+
\varepsilon\beta_3
\int_0^t
\Bigl\langle
\nabla\Bigl[
(1-|u_n^\varepsilon|^2)u_n^\varepsilon
-
(1-|u^\varepsilon|^2)u^\varepsilon
\Bigr],
\nabla w_n^\varepsilon
\Bigr\rangle_H
\,ds
\\
&\quad
-
\varepsilon\beta_4
\int_0^t
\Bigl\langle
\nabla\Bigl[
u_n^\varepsilon\times\Delta u_n^\varepsilon
-
u^\varepsilon\times\Delta u^\varepsilon
\Bigr],
\nabla w_n^\varepsilon
\Bigr\rangle_H
\,ds+\varepsilon\beta_5
\int_0^t
\Bigl\langle
\nabla\Delta\Bigl(
|u_n^\varepsilon|^2u_n^\varepsilon
-
|u^\varepsilon|^2u^\varepsilon
\Bigr),
\nabla w_n^\varepsilon
\Bigr\rangle_H
\,ds
\\
&\quad
+
\frac{\varepsilon}{2}
\sum_{j=1}^{\infty}
\int_0^t
\|\nabla(w_n^\varepsilon\times h_j)\|_{H}^2
\,ds+\varepsilon
\sum_{j=1}^{\infty}
\int_0^t
\Bigl\langle
\nabla\bigl((w_n^\varepsilon\times h_j)\times h_j\bigr),
\nabla w_n^\varepsilon
\Bigr\rangle_H
\,ds
\\
&\quad
-
\sqrt{\varepsilon}
\sum_{j=1}^{\infty}
\int_0^t
\Bigl\langle
\nabla(w_n^\varepsilon\times h_j),
\nabla w_n^\varepsilon
\Bigr\rangle_H
\,dW_j(s)\\
&\quad
:=\frac{1}{2}\|\nabla(x_n-x)\|_{H}+\sum_{i=1}^6 \frac{1}{2}L_i
\end{aligned}
\end{equation}
Define
\[
\begin{aligned}
L_1(t)
:={}&
\varepsilon\beta_3
\int_0^t
\Bigl\langle
\nabla\Bigl[
(1-|u_n^\varepsilon|^2)u_n^\varepsilon
-
(1-|u^\varepsilon|^2)u^\varepsilon
\Bigr],
\nabla w_n^\varepsilon
\Bigr\rangle_H
\,ds .
\end{aligned}
\]

Then
\[
\begin{aligned}
L_1(t)
={}&
\varepsilon\beta_3
\int_0^t
\|\nabla w_n^\varepsilon\|_H^2
\,ds-\varepsilon\beta_3
\int_0^t
\Bigl\langle
\nabla\Bigl(
|u_n^\varepsilon|^2u_n^\varepsilon-|u^\varepsilon|^2u^\varepsilon
\Bigr),
\nabla w_n^\varepsilon
\Bigr\rangle_H
\,ds .
\end{aligned}
\]
Since,
\[
|u_n^\varepsilon|^2u_n^\varepsilon
-
|u^\varepsilon|^2u^\varepsilon
=|u_n^\varepsilon|^2w_n^\varepsilon+\Bigl(|u_n^\varepsilon|^2-|u^\varepsilon|^2\Bigr)u^\varepsilon .
\]
Hence,
\[
\begin{aligned}
L_1(t) = \varepsilon\beta_3 \int_0^t \langle \nabla w_n^\varepsilon, \nabla w_n^\varepsilon \rangle_H \, ds
&- \varepsilon\beta_3 \int_0^t \Bigl\langle \nabla\bigl( |u_n^\varepsilon|^2 w_n^\varepsilon \bigr), \nabla w_n^\varepsilon \Bigr\rangle_H \, ds \\
&\quad - \varepsilon\beta_3 \int_0^t \Bigl\langle \nabla\Bigl( \bigl( |u^\varepsilon|^2 - |u_n^\varepsilon|^2 \bigr) u^\varepsilon \Bigr), \nabla w_n^\varepsilon \Bigr\rangle_H \, ds \\
&=: \varepsilon\beta_3 \int_0^t \|\nabla w_n^\varepsilon\|_H^2 \, ds + L_{1,1}(t) + L_{1,2}(t).
\end{aligned}
\]

Moreover,
\[
\begin{aligned}
|L_{1,1}(t)|
&= \varepsilon\beta_3 \left| \int_0^t \Bigl\langle |u_n^\varepsilon|^2 w_n^\varepsilon, \Delta w_n^\varepsilon \Bigr\rangle_H \, ds \right| \\
&\le \varepsilon\beta_3 \int_0^t \bigl\| |u_n^\varepsilon|^2 w_n^\varepsilon \bigr\|_{H} \, \|\Delta w_n^\varepsilon\|_{H} \, ds .
\end{aligned}
\]
Since
\(
H^1(\mathcal D)\hookrightarrow L^\infty(\mathcal D),
\)
we deduce that
\[
\begin{aligned}
|L_{1,1}(t)|
&\le
\frac{\varepsilon\beta_1}{4}
\int_0^t
\|\Delta w_n^\varepsilon\|_{H}^{2}
\,ds+\frac{C^{2}\varepsilon\beta_{3}^{2}}{\beta_{1}}
\int_0^t
\|u_n^\varepsilon\|_{H^1}^{4}
\|w_n^\varepsilon\|_{H}^{2}
\,ds .
\end{aligned}
\]
 We estimate \(L_{1,2}(t)\) as follows:
\[
\begin{aligned}
|L_{1,2}(t)|
&\le \frac{\varepsilon\beta_1}{4} \int_0^t \|\Delta w_n^\varepsilon\|_{H}^{2} \,ds + \frac{C^{2}\varepsilon\beta_{3}^{2}}{\beta_{1}} \int_0^t \|w_n^\varepsilon\|_{H}^{2} \Bigl( \|u^\varepsilon\|_{H^1}^{2} + \|u^\varepsilon\|_{H^1} \|u_n^\varepsilon\|_{H^1} \Bigr)^{2} \,ds .
\end{aligned}
\]
Consider
\[
\begin{aligned}
L_3(t)
&= \varepsilon\beta_5 \int_0^t \Bigl\langle \nabla\Bigl( |u_n^\varepsilon|^2 u_n^\varepsilon - |u^\varepsilon|^2 u^\varepsilon \Bigr), \nabla\Delta w_n^\varepsilon \Bigr\rangle_H \,ds \\
&= \varepsilon\beta_5 \int_0^t \Bigl\langle \nabla\bigl( |u_n^\varepsilon|^2 w_n^\varepsilon \bigr), \nabla\Delta w_n^\varepsilon \Bigr\rangle_H \,ds \\
&\quad + \varepsilon\beta_5 \int_0^t \Bigl\langle \nabla\Bigl( \bigl( |u_n^\varepsilon|^2 - |u^\varepsilon|^2 \bigr) u^\varepsilon \Bigr), \nabla\Delta w_n^\varepsilon \Bigr\rangle_H \,ds \\
&=: L_{3,1}(t) + L_{3,2}(t).
\end{aligned}
\]
We first estimate \(L_{3,1}(t)\). Using H\"older's inequality, Young's inequality, and the Sobolev embedding
\(H^1(\mathcal D)\hookrightarrow L^\infty(\mathcal D)\),
we obtain
\[
\begin{aligned}
|L_{3,1}(t)|
&\le \frac{\varepsilon\beta_2}{4} \int_0^t \|\nabla\Delta w_n^\varepsilon\|_{H}^{2} \,ds \\
&\quad + \frac{C^{2}\varepsilon\beta_{5}^{2}}{\beta_{2}} \int_0^t \|u_n^\varepsilon\|_{H^1}^{4} \|\nabla w_n^\varepsilon\|_{H}^{2} \,ds \\
&\quad + \frac{C^{2}\varepsilon\beta_{5}^{2}}{\beta_{2}} \int_0^t \|u_n^\varepsilon\|_{H^1}^{2} \bigl( \|u_n^\varepsilon\|_{H^1} + \|u^\varepsilon\|_{H^1} \bigr)^{2} \,ds .
\end{aligned}
\]
For the stochastic correction terms, by H\"older's inequality, Young's inequality, and the assumption
\(
\sum_{j=1}^\infty \|h_j\|_{H^3}<C_h,
\)
we obtain
\[
\begin{aligned}
&\frac{\varepsilon}{2} \sum_{j=1}^{\infty} \int_0^t \|\nabla( w_n^\varepsilon \times h_j )\|_{H}^{2} \,ds \\
&\qquad + \varepsilon \sum_{j=1}^{\infty} \int_0^t \Bigl\langle \nabla\bigl( (w_n^\varepsilon \times h_j) \times h_j \bigr), \nabla w_n^\varepsilon \Bigr\rangle_H \,ds \\
&\le \varepsilon C_h \int_0^t \|\nabla w_n^\varepsilon\|_{H}^{2} \,ds + \varepsilon C_h \int_0^t \bigl( \|u_n^\varepsilon\|_{H^1}^{2} + \|u^\varepsilon\|_{H^1}^{2} \bigr) \,ds .
\end{aligned}
\]

From the term \(L_2\), we obtain
\[
\begin{aligned}
L_2(t)
:=& -\varepsilon\beta_4 \int_0^t \Bigl\langle \nabla\Bigl[ u_n^\varepsilon \times \Delta u_n^\varepsilon - u^\varepsilon \times \Delta u^\varepsilon \Bigr], \nabla w_n^\varepsilon \Bigr\rangle_H \,ds \\
=& -\varepsilon\beta_4 \int_0^t \Bigl\langle u_n^\varepsilon \times \Delta w_n^\varepsilon, \Delta w_n^\varepsilon \Bigr\rangle_H \,ds \\
&\quad - \varepsilon\beta_4 \int_0^t \Bigl\langle w_n^\varepsilon \times \Delta u^\varepsilon, \Delta w_n^\varepsilon \Bigr\rangle_H \,ds \\
=:& L_{2,1}(t) + L_{2,2}(t).
\end{aligned}
\]
Using the identity
\[
(a \times b) \cdot b = 0,
\]
we deduce that
\[
L_{2,1}(t) = 0.
\]

Moreover, by H\"older's inequality, Young's inequality, and the Sobolev embedding
\( H^1(\mathcal D) \hookrightarrow L^\infty(\mathcal D) \),
we obtain
\[
\begin{aligned}
|L_{2,2}(t)|
&\le \varepsilon\beta_4 \int_0^t \| w_n^\varepsilon \times \Delta u^\varepsilon \|_{H} \, \| \Delta w_n^\varepsilon \|_{H} \,ds \\
&\le \varepsilon\beta_4 \int_0^t \| w_n^\varepsilon \|_{L^\infty} \| \Delta u^\varepsilon \|_{H} \, \| \Delta w_n^\varepsilon \|_{H} \,ds \\
&\le \frac{\varepsilon\beta_1}{4} \int_0^t \| \Delta w_n^\varepsilon \|_{H}^{2} \,ds + \frac{C^{2}\varepsilon\beta_{4}^{2}}{\beta_{1}} \int_0^t \bigl(\| u_n^\varepsilon \|_{H^1}^{2}+\|u^\varepsilon\|_{H^1}^2\bigr) \| u^\varepsilon \|_{H^2}^{2} \,ds .
\end{aligned}
\]

Taking all the above estimates together, and using
\[
\|w_n^\varepsilon\|_{H} \le \|u_n^\varepsilon\|_{H} + \|u^\varepsilon\|_{H} \le \|u_n^\varepsilon\|_{H^1} + \|u^\varepsilon\|_{H^1},
\]
we obtain
\[
\begin{aligned}
\frac12 \|\nabla w_n^\varepsilon(t)\|_{H}^{2}
&\le \frac12 \|\nabla(x_n - x)\|_{H}^{2} \\
&\quad + \int_0^t \Biggl( \varepsilon\beta_3 + \varepsilon C_h + \frac{C^{2}\varepsilon\beta_{5}^{2}}{\beta_{2}} \|u_n^\varepsilon(s)\|_{H^1}^{4} \Biggr) \|\nabla w_n^\varepsilon(s)\|_{H}^{2} \,ds \\
&\quad + \frac{C^{2}\varepsilon\beta_{3}^{2}}{\beta_{1}} \int_0^t \Bigl( \|u_n^\varepsilon(s)\|_{H^1} + \|u^\varepsilon(s)\|_{H^1} \Bigr)^{2} \Bigl( \|u^\varepsilon(s)\|_{H^1}^{2} + \|u^\varepsilon(s)\|_{H^1} \|u_n^\varepsilon(s)\|_{H^1} \Bigr)^{2} \,ds \\
&\quad + \frac{C^{2}\varepsilon\beta_{3}^{2}}{\beta_{1}} \int_0^t \|u_n^\varepsilon(s)\|_{H^1}^{4} \Bigl( \|u_n^\varepsilon(s)\|_{H^1} + \|u^\varepsilon(s)\|_{H^1} \Bigr)^{2} \,ds + \varepsilon C_h \int_0^t \Bigl( \|u_n^\varepsilon(s)\|_{H^1}^{2} + \|u^\varepsilon(s)\|_{H^1}^{2} \Bigr) \,ds \\
&\quad + \frac{C^{2}\varepsilon\beta_{5}^{2}}{\beta_{2}} \int_0^t \|u_n^\varepsilon(s)\|_{H^1}^{2} \Bigl( \|u_n^\varepsilon(s)\|_{H^1} + \|u^\varepsilon(s)\|_{H^1} \Bigr)^{2} \,ds \\
&\quad - \sqrt{\varepsilon} \sum_{j=1}^{\infty} \int_0^t \Bigl\langle \nabla( w_n^\varepsilon(s) \times h_j ), \nabla w_n^\varepsilon(s) \Bigr\rangle_H \,dW_j(s).
\end{aligned}
\]

For $M > 0$, define stopping times
\[
\vartheta_{\varepsilon,M}^1 = \inf \left\{ t \ge 0 : \int_0^t \|u^\varepsilon(s)\|_{H^2}^2 >M,\,\text{or } \|u^\varepsilon(s)\|_{H^1} \,ds > M \right\}
\]
and
\[
\vartheta_{\varepsilon,M}^2 = \inf \left\{ t \ge 0 : \|u_n^\varepsilon(t)\|_{H^1} > M \right\}.
\]
Let $\vartheta_{\varepsilon,M} = \vartheta_{\varepsilon,M}^1 \wedge \vartheta_{\varepsilon,M}^2$. Then Gronwall's inequality yields
\[
\begin{aligned}
\sup_{0 \le t \le T \wedge \vartheta_{\varepsilon,M}^{\varepsilon}} \|\nabla w_n^\varepsilon(t)\|_{H}^{2}
&\le \Biggl[ \|\nabla(x_n - x)\|_{H}^{2} + \varepsilon C_{\beta_1,\beta_3} T M^{6} + \varepsilon C_h T M^{2} + \varepsilon C_{\beta_2,\beta_5} T M^{3} + \sup_{0 \le t \le T \wedge \vartheta_{\varepsilon,M}^{\varepsilon}} |L_6(t)| \Biggr] \\
&\qquad \cdot \exp\Bigl\{ \bigl( \varepsilon \beta_3 + \varepsilon C_h + \varepsilon C_{\beta_2,\beta_5} M^{4} \bigr)T \Bigr\} \\
&=: \Biggl[ \|\nabla(x_n - x)\|_{H}^{2} + \varepsilon C_{h,\beta_1,\beta_2,\beta_3,\beta_5}^{M,T} + \sup_{0 \le t \le T \wedge \vartheta_{\varepsilon,M}^{\varepsilon}} |L_6(t)| \Biggr]\exp\bigl\{ C_{M,h,\beta_2,\beta_3,\beta_5}^{\varepsilon} T \bigr\}.
\end{aligned}
\]
We define the constant \( C_{h,\beta}^{M,T} := C_{h,\beta_1,\beta_2,\beta_3,\beta_5}^{M,T} \), and the constant \( C_{M,h}^{\varepsilon,\beta} := C_{M,h,\beta_2,\beta_3,\beta_5}^{\varepsilon}T \) for the simplification of the notations.

Now consider
\[
L_6(t)
:=
- \sqrt{\varepsilon}
\sum_{j=1}^{\infty}
\int_0^t
\Bigl\langle
\nabla( w_n^\varepsilon(s) \times h_j ),
\nabla w_n^\varepsilon(s)
\Bigr\rangle_H
\,dW_j(s).
\]
Moreover,
\[
\begin{aligned}
&
\sum_{j=1}^{\infty}
\int_0^{T \wedge \vartheta_{\varepsilon,M}^{\varepsilon}}
\Bigl|
\Bigl\langle
\nabla( w_n^\varepsilon(s) \times h_j ),
\nabla w_n^\varepsilon(s)
\Bigr\rangle_H
\Bigr|^{2}
\,ds
\\
&\le
C_{h,M}
\int_0^{T \wedge \vartheta_{\varepsilon,M}^{\varepsilon}}
\|\nabla w_n^\varepsilon(s)\|_{H}^{2}
\,ds .
\end{aligned}
\]
By the Burkholder--Davis--Gundy inequality, we get
\[
\begin{aligned}
\mathbb{E}
\Biggl[
\sup_{0 \le t \le T \wedge \vartheta_{\varepsilon,M}^{\varepsilon}}
|L_6(t)|^p
\Biggr]
&\le
C_p \, \varepsilon^{p/2}
\,
\mathbb{E}
\Biggl[
\Biggl(
\sum_{j=1}^{\infty}
\int_0^{T \wedge \vartheta_{\varepsilon,M}^{\varepsilon}}
\Bigl|
\Bigl\langle
\nabla( w_n^\varepsilon(s) \times h_j ),
\nabla w_n^\varepsilon(s)
\Bigr\rangle_H
\Bigr|^2
\, ds
\Biggr)^{p/2}
\Biggr].
\end{aligned}
\]
Consequently, we can write
\[
\begin{aligned}
& \Biggl( \mathbb{E} \Bigl( \sup_{0 \le t \le T \wedge \vartheta_{\varepsilon,M}^{\varepsilon}} \|\nabla w_n^\varepsilon(t)\|_{H}^{2p} \Bigr) \Biggr)^{\frac{2}{p}} \\
&\qquad \le 3 \exp\bigl\{ 2 C_{M,h}^{\varepsilon,\beta} T \bigr\} \|\nabla(x_n - x)\|_{H}^{4} + 3 \bigl( C_{h,\beta}^{M,T} \bigr)^{2} \exp\bigl\{ 2 C_{M,h}^{\varepsilon,\beta} T \bigr\} \\
&\qquad \quad + 3 \exp\bigl\{ 2 C_{M,h}^{\varepsilon,\beta} T \bigr\} C_{h,M}^{2} p \varepsilon \int_0^{T \wedge \vartheta_{\varepsilon,M}^{\varepsilon}} \Biggl( \mathbb{E} \bigl[ \|\nabla w_n^\varepsilon(s)\|_{H}^{2p} \bigr] \Biggr)^{\frac{2}{p}} \,ds .
\end{aligned}
\]
Applying Gronwall's inequality, we obtain
\[
\begin{aligned}
\Biggl( \mathbb{E} \Bigl( \sup_{0 \le t \le T \wedge \vartheta_{\varepsilon,M}^{\varepsilon}} \|\nabla w_n^\varepsilon(t)\|_{H}^{2p} \Bigr) \Biggr)^{\frac{2}{p}} & \le \Biggl[ 3 \exp\bigl\{ 2 C_{M,h}^{\varepsilon,\beta} T \bigr\} \|\nabla(x_n - x)\|_{H}^{4}+ 3 \bigl( C_{h,\beta}^{M,T} \bigr)^{2} \exp\bigl\{ 2 C_{M,h}^{\varepsilon,\beta} T \bigr\} \Biggr] \\
&\qquad \qquad \times \exp\Biggl\{ 3 \exp\bigl\{ 2 C_{M,h}^{\varepsilon,\beta} T \bigr\} C_{h,M}^{2} p \varepsilon T \Biggr\}.
\end{aligned}
\]
Taking \(p = \frac{2}{\varepsilon}\), and proceeding exactly as in the proof of Lemma~\ref{lemma2}, we obtain
\[
\sup_{0<\varepsilon\le1} \varepsilon \log \mathbb{P} \Biggl( \sup_{0\le s\le T\wedge \vartheta_{\varepsilon,M}^{\varepsilon}} \|\nabla w_n^\varepsilon(s)\|_{H}^2 > \delta \Biggr) \longrightarrow -\infty, \qquad \text{as } n \to \infty .
\]
For any \(R > 0\), Lemma~\ref{lemma1} implies that there exists \(M > 0\) such that, for any \(0 < \varepsilon \le 1\) and \(n \ge 1\),
\[
\mathbb{P} \Bigl( \|u^\varepsilon\|_{H^1}^{H^2}(T) > M \Bigr) \le e^{-K/\varepsilon}, \qquad
\mathbb{P} \Bigl( \|u_n^\varepsilon\|_{H^1}^{H^2}(T) > M \Bigr) \le e^{-K/\varepsilon}.
\]

Hence, for such \(M\), there exists a positive integer \(N\) such that, for all \(n \ge N\),
\[
\begin{aligned}
& \sup_{0<\varepsilon\le1} \varepsilon \log \mathbb{P} \Biggl( \sup_{0\le t\le T} \|\nabla w_n^\varepsilon(t)\|_{H}^{2} > \delta, \;
\|u^\varepsilon\|_{H^1}^{H^2}(T)\le M, \;
\|u_n^\varepsilon\|_{H^1}^{H^2}(T) \le M \Biggr) \\
&\le \sup_{0<\varepsilon\le1} \varepsilon \log \mathbb{P} \Biggl( \sup_{0\le t\le T\wedge \vartheta_{\varepsilon,M}^{\varepsilon}} \|\nabla w_n^\varepsilon(t)\|_{H}^{2} > \delta \Biggr) \le -K .
\end{aligned}
\]

Combining the above estimates, we deduce that
\[
\begin{aligned}
& \mathbb{P} \Biggl( \sup_{0\le t\le T} \|\nabla w_n^\varepsilon(t)\|_{H}^{2} > \delta \Biggr) \\
&\le \mathbb{P} \Bigl( \|u^\varepsilon\|_{H^1}^{H^2}(T) > M \Bigr) + \mathbb{P} \Bigl( \|u_n^\varepsilon\|_{H^1}^{H^2}(T) > M \Bigr) \\
&\quad + \mathbb{P} \Biggl( \sup_{0\le t\le T} \|\nabla w_n^\varepsilon(t)\|_{H}^{2} > \delta, \;
\|u^\varepsilon\|_{H^1}^{H^2}(T) \le M, \;
\|u_n^\varepsilon\|_{H^1}^{H^2}(T) \le M \Biggr) \\
&\le 3 e^{-K/\varepsilon}.
\end{aligned}
\]
Since
\[
\mathbb{P}\!\left(
\sup_{0\le t\le T}\|\nabla w_n^\varepsilon(t)\|_{H}^{2}
>\delta
\right)
\le 3e^{-K/\varepsilon},
\]
and, by Lemma~\ref{lemma2},
\[
\mathbb{P}\!\left(
\sup_{0\le t\le T}\|w_n^\varepsilon(t)\|_{H}^{2}
>\delta
\right)
\le 3e^{-K/\varepsilon},
\]
it follows from the definition of the \(H^{1}\)-norm and the union bound that
\begin{align*}
\mathbb{P}\!\left(
\sup_{0\le t\le T}\|w_n^\varepsilon(t)\|_{H^{1}}^{2}
>2\delta
\right)
&=
\mathbb{P}\!\left(
\sup_{0\le t\le T}
\bigl(
\|w_n^\varepsilon(t)\|_{H}^{2}+\|\nabla w_n^\varepsilon(t)\|_{H}^{2}\bigr)>2\delta\right) \\
&\le
\mathbb{P}\!\left(
\sup_{0\le t\le T}\|w_n^\varepsilon(t)\|_{H}^{2}
>\delta\right)
+\mathbb{P}\!\left(
\sup_{0\le t\le T}\|\nabla w_n^\varepsilon(t)\|_{H}^{2}
>\delta\right) \\
&\le 6e^{-K/\varepsilon}.
\end{align*}
Since \(K>0\) is arbitrary, the proof is complete.

\end{proof}

\begin{lemma}{\label{lemma4}}
   For any fixed \(n\in\mathbb{N}\) and any \(\delta > 0\), we have
\[
\lim_{\varepsilon \to 0} \,
\varepsilon \log \mathbb{P} \left( 
\sup_{0 \le t \le T} 
\|u_n^\varepsilon(t) - v_n^\varepsilon(t)\|_{H}^2 > \delta 
\right) = -\infty.
\]
\end{lemma}
\begin{proof}
  Let \(d_n := u_n^{\varepsilon} - v_n^{\varepsilon}\), where \(u_n^{\varepsilon}\) and \(v_n^{\varepsilon}\) are the solutions of \eqref{eq:ito_LLBar} and \eqref{eq:Yeps}, respectively, with the same initial value \(x_n\). Then, by It\^o's formula, we have
\[
\begin{aligned}
\| d_n(t)\|_{H}^{2}
&= 2\varepsilon\beta_1 \int_0^t \bigl\langle d_n, \Delta u_n^\varepsilon \bigr\rangle_H \,ds- 2\varepsilon\beta_2 \int_0^t \bigl\langle d_n, \Delta^2 u_n^\varepsilon \bigr\rangle_H \,ds \\
&\quad + 2\varepsilon\beta_3 \int_0^t \Bigl\langle d_n, (1-|u_n^\varepsilon|^2)u_n^\varepsilon \Bigr\rangle_H \,ds  - 2\varepsilon\beta_4 \int_0^t \Bigl\langle d_n, u_n^\varepsilon \times \Delta u_n^\varepsilon \Bigr\rangle_H \,ds \\
&\quad + 2\varepsilon\beta_5 \int_0^t \Bigl\langle d_n, \Delta \bigl( |u_n^\varepsilon|^2 u_n^\varepsilon \bigr) \Bigr\rangle_H \,ds - \varepsilon \sum_{j=1}^{\infty} \int_0^t \Bigl\langle d_n, G_j(u_n^\varepsilon) \times h_j \Bigr\rangle_H \,ds \\
&\quad + \varepsilon \sum_{j=1}^{\infty} \int_0^t \| G_j(u_n^\varepsilon) - G_j(v_n^\varepsilon) \|_{H}^{2} \,ds  + 2\sqrt{\varepsilon} \sum_{j=1}^{\infty} \int_0^t \Bigl\langle d_n, G_j(u_n^\varepsilon) - G_j(v_n^\varepsilon) \Bigr\rangle_H \,dW_j(s)\\
&
:= \sum_{k=1}^{8} J_k(t),
\end{aligned}
\]
Consider the above terms \(J_k\), \(k=1,2,\dots,8\), one by one as follows:

\(\textbf{\underline{\(J_1\) estimation:}}\)
\[
\begin{aligned}
J_1(t)
&= 2\varepsilon\beta_1 \int_0^t \langle d_n, \Delta u_n^\varepsilon \rangle_H \,ds \\
&= -2\varepsilon\beta_1 \int_0^t \|\nabla d_n\|_H^2 \,ds + 2\varepsilon\beta_1 \int_0^t \langle d_n, \Delta v_n^\varepsilon \rangle_H \,ds \\
&\le -2\varepsilon\beta_1 \int_0^t \|\nabla d_n\|_H^2 \,ds + \frac{\varepsilon \beta_2}{4} \int_0^t \|\Delta d_n\|_{H}^2 \,ds \\
&\qquad + \varepsilon C_{\beta_2,\beta_1} \int_0^t \|d_n\|_{H}^2 \,ds.
\end{aligned}
\]

\(\textbf{\underline{\(J_2\) estimation:}}\)

\[
\begin{aligned}
J_2(t)
&= -2\varepsilon\beta_2 \int_0^t \langle d_n, \Delta^2 u_n^\varepsilon \rangle_H \,ds \\
&= -2\varepsilon\beta_2 \int_0^t \|\Delta d_n\|_H^2 \,ds - 2\varepsilon\beta_2 \int_0^t \langle d_n, \Delta^2 v_n^\varepsilon \rangle_H \,ds \\
&\le -2\varepsilon\beta_2 \int_0^t \|\Delta d_n\|_H^2 \,ds +\frac{\varepsilon\beta_2}{4}\int_0^t\|\Delta d_n\|_{H}^2\,ds+ \varepsilon C_{\beta_2} \int_0^t \|\Delta v_n^\varepsilon\|_H^2 \,ds.
\end{aligned}
\]

\( \textbf{\underline{\(J_3\) estimation:}}\)

\[
\begin{aligned}
J_3(t)
&= 2\varepsilon\beta_3 \int_0^t \Bigl\langle d_n, (1-|u_n^\varepsilon|^2)u_n^\varepsilon \Bigr\rangle_H \,ds, \\
&\le 2\varepsilon\beta_3 \int_0^t \|d_n\|_H \, \|(1-|u_n^\varepsilon|^2)u_n^\varepsilon\|_H \,ds \\
&\le C\varepsilon \int_0^t \|d_n\|_H^2 \,ds + \varepsilon C_{\beta_3} \int_0^t \bigl( 1 + \|u_n^\varepsilon\|_{H^1}^6 \bigr) \,ds .
\end{aligned}
\]

\(
\textbf{\underline{\(J_4\) estimation:}}
\)

\[
\begin{aligned}
J_4(t)
&= -2\varepsilon\beta_4 \int_0^t \Bigl\langle d_n, u_n^\varepsilon \times \Delta u_n^\varepsilon \Bigr\rangle_H \,ds, \\
&\le 2\varepsilon\beta_4 \int_0^t \|d_n\|_H \, \| u_n^\varepsilon \times \Delta u_n^\varepsilon \|_H \,ds \\
&\le C\varepsilon \int_0^t \|d_n\|_H^2 \,ds + C_{\beta_3}\varepsilon \int_0^t \|u_n^\varepsilon\|_{H^1}^2 \|\Delta u_n^\varepsilon\|_{H}^2 \,ds .
\end{aligned}
\]

\(
\textbf{\underline{\(J_5\) estimation:}}
\)
  Using \eqref{eq:aux3} we obtain
\[
\begin{aligned}
J_5(t)
&= 2\varepsilon\beta_5 \int_0^t \Bigl\langle d_n, \Delta\bigl(|u_n^\varepsilon|^2 u_n^\varepsilon\bigr) \Bigr\rangle_H \,ds \\
&\le \frac{\varepsilon\beta_2}{4}\int_0^t \|\Delta d_n\|_H^2 \,ds + \varepsilon C_{\beta_2,\beta_5}\|u_n^\varepsilon\|_{H^1}^6 \,ds .
\end{aligned}
\]

\(
\textbf{\underline{\(J_6\) estimation:}}
\)
\[
\begin{aligned}
J_6(t)
&= -\varepsilon \sum_{j=1}^{\infty} \int_0^t \Bigl\langle d_n, G_j(u_n^\varepsilon) \times h_j \Bigr\rangle_H \,ds, \\
&\le \varepsilon \sum_{j=1}^{\infty} \int_0^t \|d_n\|_H \, \| G_j(u_n^\varepsilon) \times h_j \|_H \,ds \\
&\le C \varepsilon \int_0^t \|d_n\|_H^2 \,ds + C_h \varepsilon \int_0^t \bigl( 1 + \|u_n^\varepsilon\|_{H^1}^2 \bigr) \,ds .
\end{aligned}
\]

\(
\textbf{\underline{\(J_7\) estimation:}}
\)
\[
\begin{aligned}
J_7(t)
&= \varepsilon \sum_{j=1}^{\infty} \int_0^t \| G_j(u_n^\varepsilon) - G_j(v_n^\varepsilon) \|_{H}^2 \,ds \\
&= \varepsilon \sum_{j=1}^{\infty} \int_0^t \| d_n \times h_j \|_H^2 \,ds \\
&\le C_h \varepsilon \int_0^t \| d_n \|_H^2 \,ds .
\end{aligned}
\]

\(
\textbf{\underline{\(J_8\) estimation:}}
\)
\[
\begin{aligned}
J_8(t)
&= 2\sqrt{\varepsilon} \sum_{j=1}^{\infty} \int_0^t \Bigl\langle d_n, G_j(u_n^\varepsilon) - G_j(v_n^\varepsilon) \Bigr\rangle_H \,dW_j(s) \\
&= 2\sqrt{\varepsilon} \sum_{j=1}^{\infty} \int_0^t \bigl\langle d_n, -(d_n \times h_j) \bigr\rangle_H \,dW_j(s) = 0,
\end{aligned}
\]
Collecting all estimates, we obtain
\[
\begin{aligned}
\|d_n(t)\|_H^2
&\le \int_0^t \bigl( \varepsilon C_h + \varepsilon C + \varepsilon C_{\beta_1,\beta_2} \bigr) \|d_n\|_{H}^2 \, ds \\
&\quad +\int_0^t \Bigl( \varepsilon C_{\beta_2} \|v_n^\varepsilon\|_{H^2}^2 + \varepsilon C_{\beta_3} \bigl(1 + \|u_n^\varepsilon\|_{H^1}^6\bigr)+ \varepsilon C_{\beta_3} \|u_n^\varepsilon\|_{H^2}^4 + \varepsilon C_{\beta_2,\beta_5} \|u_n^\varepsilon\|_{H^1}^6 \Bigr) \,ds .
\end{aligned}
\]
For \(M > 0\), define the stopping times
\[
 \sigma_{M}^{\varepsilon,1}
:= \inf\Bigl\{ t \ge 0 : \|v_n^\varepsilon(t)\|_{H^2}^{2} > M \Bigr\},
\]
and
\[
 \sigma_{M}^{\varepsilon,2}
:= \inf\Biggl\{ t \ge 0 : \int_0^t \|u_n^\varepsilon(s)\|_{H^2}^{2} \, ds > M, \ \text{or} \ \|u_n^\varepsilon(t)\|_{H^1}^{2} > M \Biggr\}.
\]
Let
\[
 \sigma_{M}^{\varepsilon} := \sigma_{M}^{\varepsilon,1} \wedge \sigma_{M}^{\varepsilon,2}.
\]
Applying Gronwall's inequality up to the stopping time \( \sigma_{M}^{\varepsilon}\), we obtain
\[
\begin{aligned}
\sup_{0\le t\le T\wedge \sigma_{M}^{\varepsilon}}
\|d_n(t)\|_H^2
&\le
\Bigl( \varepsilon C_{\beta_2} M^2 + \varepsilon C_{\beta_3}(1 + M^6) \\
&\qquad\qquad + \varepsilon C_{\beta_3} M^4 + \varepsilon C_{\beta_2,\beta_5} M^6 \Bigr) T \\
&\quad \times \exp\Biggl( \int_0^{T\wedge \sigma_{M}^{\varepsilon}} \bigl( \varepsilon C_h + \varepsilon C + \varepsilon C_{\beta_1,\beta_2} \bigr) \,ds \Biggr).
\end{aligned}
\]

Define
\[
C_{\beta,\varepsilon}^{M,T}
:= \Bigl( \varepsilon C_{\beta_2} M^2 + \varepsilon C_{\beta_3}(1 + M^6) + \varepsilon C_{\beta_3} M^4 + \varepsilon C_{\beta_2,\beta_5} M^6 \Bigr) T,
\]
and
\[
D_{h,\beta}^{\varepsilon,T} := (\varepsilon C_h + \varepsilon C + \varepsilon C_{\beta_1,\beta_2})T.
\]

Taking expectation and using the identity \(d_n = u_n^\varepsilon - v_n^\varepsilon\), we obtain
\[
\mathbb{E} \Biggl( \sup_{0\le s\le T\wedge \sigma_{M}^{\varepsilon}} \|u_n^\varepsilon(s) - v_n^\varepsilon(s)\|_{H}^{2} \Biggr)
\le C_{\beta_2,\beta_3,\beta_5,\varepsilon}^{M,T} \exp\bigl( D_{h,\beta_1,\beta_2}^{\varepsilon,T} \bigr).
\]

Moreover,
\[
\Biggl( \mathbb{E} \Biggl[ \sup_{0\le s\le T\wedge \sigma_{M}^{\varepsilon} }\|u_n^\varepsilon(s) - v_n^\varepsilon(s)\|_{H}^{2p} \Biggr] \Biggr)^{\frac{2}{p}}
\le 2 \bigl( C_{\beta_2,\beta_3,\beta_5,\varepsilon}^{M,T} \bigr)^2 \exp\bigl\{ 2 D_{h,\beta_1,\beta_2}^{\varepsilon,T} \bigr\}.
\]
Taking \(p=\frac{2}{\varepsilon}\), and proceeding exactly as in the proofs of Lemma~\ref{lemma2} and Lemma~\ref{lemma3}, we obtain
\[
\sup_{0<\varepsilon\le1}
\varepsilon
\log
\mathbb{P}
\Biggl(
\sup_{0\le s\le T\wedge \sigma_{M}^{\varepsilon}}
\|u_n^\varepsilon(s)-v_n^\varepsilon(s)\|_{H}^{2}
>\delta
\Biggr)
\longrightarrow -\infty,
\qquad \text{as } \varepsilon\to0 .
\]
For any \(K > 0\), using the above estimate, Lemma~\ref{lemma1}, and \cite[Lemma 5.2]{ROCKNER2012716}, there exists a constant \(M > 0\) such that, for any \(0 < \varepsilon \le 1\),
\[
\mathbb{P} \Bigl( \|u_n^\varepsilon\|_{H^1}^{H^2}(T) > M \Bigr) \le e^{-K/\varepsilon},
\]
and
\[
\mathbb{P} \Biggl( \sup_{0 \le t \le T} \|v_n^\varepsilon(t)\|_{H^2}^{2} > M \Biggr) \le e^{-K/\varepsilon}.
\]

Moreover, by the above estimate, there exists \(0 < \varepsilon_0 < 1\) such that, for all \(\varepsilon \le \varepsilon_0\),
\[
\begin{aligned}
& \varepsilon \log \mathbb{P} \Biggl( \sup_{0 \le s \le T} \|u_n^\varepsilon(s) - v_n^\varepsilon(s)\|_{H}^{2} > \delta, \;
\|u_n^\varepsilon\|_{H^1}^{H^2}(T) \le M, \\
&\qquad\qquad\qquad\qquad \sup_{0 \le t \le T} \|v_n^\varepsilon(t)\|_{H^2}^{2} \le M \Biggr) \\
&\le \varepsilon \log \mathbb{P} \Biggl( \sup_{0 \le s \le T \wedge \sigma_{M}^{\varepsilon}} \|u_n^\varepsilon(s) - v_n^\varepsilon(s)\|_{H}^{2} > \delta \Biggr) \le -K .
\end{aligned}
\]

Combining these estimates, one can obtain
\[
\mathbb{P}
\Biggl(
\sup_{0\le t\le T}
\|u_n^\varepsilon(t)-v_n^\varepsilon(t)\|_{H}^{2}
>\delta
\Biggr)
\le
3e^{-K/\varepsilon},
\]
for all \(\varepsilon\le\varepsilon_0\). Hence, the proof is completed.
\end{proof}

\begin{lemma}{\label{lemma5}}
   For any fixed \(n\in\mathbb{N}\) and any \(\delta > 0\), we have
\[
\lim_{\varepsilon \to 0} \,
\varepsilon \log \mathbb{P} \left( 
\sup_{0 \le t \le T} 
\|u_n^\varepsilon(t) - v_n^\varepsilon(t)\|_{H^1}^2 >2 \delta 
\right) = -\infty.
\]
\end{lemma}

\begin{proof}
Let \(d_n := u_n^{\varepsilon}-v_n^{\varepsilon}\), where \(u_n^{\varepsilon}\) and \(v_n^{\varepsilon}\) are the solutions of \eqref{eq:ito_LLBar} and \eqref{eq:Yeps}, respectively, with the same initial value \(x_n\).
Then, by It\^o's formula,
\begin{equation}
\begin{aligned}
\|\nabla d_n(t)\|_{H}^{2}
&= 2\varepsilon\beta_1\int_0^t \langle \nabla d_n,\nabla \Delta u_n^{\varepsilon}\rangle_{H}\,ds
- 2\varepsilon\beta_2\int_0^t \langle \nabla d_n,\nabla \Delta^{2}u_n^{\varepsilon}\rangle_{H}\,ds \\
&\quad + 2\varepsilon\beta_3\int_0^t \Bigl\langle \nabla d_n,\,
\nabla\bigl((1-|u_n^\varepsilon|^2)u_n^\varepsilon\bigr) \Bigr\rangle_{H}\,ds 
- 2\varepsilon\beta_4\int_0^t \Bigl\langle \nabla d_n,\,
\nabla\bigl(u_n^\varepsilon\times\Delta u_n^\varepsilon\bigr) \Bigr\rangle_{H}\,ds \\
&\quad + 2\varepsilon\beta_5\int_0^t \Bigl\langle \nabla d_n,\,
\nabla\Delta\bigl(|u_n^\varepsilon|^2u_n^\varepsilon\bigr) \Bigr\rangle_{H}\,ds 
- \varepsilon\sum_{j=1}^{\infty}\int_0^t \Bigl\langle \nabla d_n,\,
\nabla\bigl(G_j(u_n^\varepsilon)\times h_j\bigr) \Bigr\rangle_{H}\,ds \\
&\quad + \varepsilon\sum_{j=1}^{\infty}\int_0^t \|\nabla(G_j(u_n^\varepsilon)-G_j(v_n^\varepsilon))\|_{H}^{2}\,ds  + 2\sqrt{\varepsilon}\sum_{j=1}^{\infty}\int_0^t \Bigl\langle \nabla d_n,\,
\nabla(G_j(u_n^\varepsilon)-G_j(v_n^\varepsilon)) \Bigr\rangle_{H}\,dW_j(s) \\
&= \sum_{i=1}^{8} K_i(t),
\end{aligned}
\end{equation}

By integration by parts, we obtain
\begin{equation}
\begin{aligned}
2\varepsilon\beta_1\int_0^t \langle \nabla d_n, \nabla \Delta u_n^{\varepsilon} \rangle_H \,ds 
&= -2\varepsilon\beta_1\int_0^t \|\Delta d_n\|_H^2 \,ds
+ 2\varepsilon\beta_1\int_0^t \langle \nabla d_n, \nabla \Delta v_n^{\varepsilon} \rangle_H \,ds \\
&\leq -2\varepsilon\beta_1\int_0^t \|\Delta d_n\|_H^2 \,ds
+ \frac{\varepsilon\beta_1}{4}\int_0^t \|\Delta d_n\|_H^2 \,ds \\
&\quad + C\varepsilon\beta_1\int_0^t \| \Delta v_n^{\varepsilon}\|_H^2 \,ds .
\end{aligned}
\end{equation}
Similarly, for the term \(K_2\) we have
\begin{equation}
\begin{aligned}
K_2
&= -2\varepsilon\beta_2\int_0^t \langle \nabla d_n, \nabla \Delta^2 u_n^{\varepsilon} \rangle_H \,ds \\
&= -2\varepsilon\beta_2\int_0^t \langle \nabla d_n, \nabla \Delta^2 d_n \rangle_H \,ds
- 2\varepsilon\beta_2\int_0^t \langle \nabla d_n, \nabla \Delta^2 v_n^{\varepsilon} \rangle_H \,ds .
\end{aligned}
\end{equation}

Using integration by parts, the first term yields
\[
-2\varepsilon\beta_2 \langle \nabla d_n, \nabla \Delta^2 d_n \rangle_H = -2\varepsilon\beta_2 \|\nabla \Delta d_n\|_H^2 .
\]

Then, applying Young's inequality to the second term, we obtain
\begin{equation}
\begin{aligned}
K_2
&= -2\varepsilon\beta_2\int_0^t \|\nabla\Delta d_n\|_H^2 \,ds
- 2\varepsilon\beta_2\int_0^t \langle \nabla d_n, \nabla \Delta^2 v_n^{\varepsilon} \rangle_H \,ds \\
&\leq -2\varepsilon\beta_2\int_0^t \|\nabla\Delta d_n\|_H^2 \,ds
+ \frac{\varepsilon\beta_2}{4}\int_0^t \|\nabla\Delta d_n\|_H^2 \,ds \\
&\quad + C\varepsilon\beta_2\int_0^t \|\nabla  v_n^{\varepsilon}\|_H^2 \,ds .
\end{aligned}
\end{equation}
For the term \(K_3\) involving \(\beta_3\), we have
\begin{equation}
K_3
=
2\varepsilon\beta_3
\int_0^t
\Bigl\langle
\nabla d_n,\,
\nabla\bigl((1-|u_n^\varepsilon|^2)u_n^\varepsilon\bigr)
\Bigr\rangle_H ds .
\end{equation}

Using Young's inequality and the embedding \(H^1 \hookrightarrow L^\infty\), we obtain
\begin{equation}
\begin{aligned}
K_3
&\leq
\frac{\varepsilon\beta_1}{4}
\int_0^t
\|\Delta d_n\|_{H}^{2}\,ds
+
\varepsilon C_{\beta_1,\beta_3}\int_0^t
\bigl(1+\|u_n^\varepsilon\|_{H^1}^{6}\bigr)\,ds .
\end{aligned}
\end{equation}

For the term \(K_4\) involving \(\beta_4\), we have
\begin{equation}
K_4
=
-2\varepsilon\beta_4
\int_0^t
\Bigl\langle
\nabla d_n,\,
\nabla\bigl(u_n^\varepsilon\times\Delta u_n^\varepsilon\bigr)
\Bigr\rangle_H ds .
\end{equation}

Using Young's inequality and the embedding \(H^1\hookrightarrow L^\infty\), we obtain
\begin{equation}
\begin{aligned}
K_4
&\le
\frac{\varepsilon\beta_1}{4}
\int_0^t
\|\Delta d_n\|_H^2\,ds +C\varepsilon\frac{\beta_4^2}{\beta_1}
\int_0^t
\|u_n^\varepsilon\|_{H^2}^2\|u_n^\varepsilon\|_{H^1}^2\,ds .
\end{aligned}
\end{equation}

For the term \(K_5\) involving \(\beta_5\), we have
\begin{equation}
K_5
=
2\varepsilon\beta_5
\int_0^t
\Bigl\langle
\nabla d_n,\,
\nabla\Delta\bigl(|u_n^\varepsilon|^2u_n^\varepsilon\bigr)
\Bigr\rangle_H ds .
\end{equation}

Recall the identity
\[
\nabla(|v|^2v)
=
2v(v\cdot\nabla v)
+
|v|^2\nabla v .
\]

Thus, for \(u_n^\varepsilon\), we have
\[
\nabla\bigl(|u_n^\varepsilon|^2u_n^\varepsilon\bigr)
=
2u_n^\varepsilon
(u_n^\varepsilon\cdot\nabla u_n^\varepsilon)
+
|u_n^\varepsilon|^2\nabla u_n^\varepsilon .
\]

Using Young's inequality and the embedding \(H^1\hookrightarrow L^\infty\), we obtain
\begin{equation}
\begin{aligned}
|K_5|
&\le
\frac{\varepsilon\beta_2}{4}
\int_0^t
\|\nabla\Delta d_n\|_H^2\,ds
+
C\varepsilon\frac{\beta_5^2}{\beta_2}
\int_0^t
\|u_n^\varepsilon\|_{H^1}^6\,ds .
\end{aligned}
\end{equation}
For the term
\begin{equation}
\varepsilon
\sum_{j=1}^{\infty}
\int_0^t
\Bigl\langle
\nabla d_n,\,
\nabla G_j(u_n^\varepsilon)\times h_j
\Bigr\rangle_H ds ,
\end{equation}
using Young's inequality, we obtain
\begin{equation}
\begin{aligned}
\varepsilon
\sum_{j=1}^{\infty}
\int_0^t
\Bigl\langle
\nabla d_n,\,
\nabla G_j(u_n^\varepsilon)\times h_j
\Bigr\rangle_H ds
&\le
\varepsilon
\int_0^t
\|\nabla d_n\|_H^2\,ds \\
&\quad
+
\varepsilon C_h
\int_0^t
\bigl(1+\|u_n^\varepsilon\|_{H^1}^2\bigr)\,ds .
\end{aligned}
\end{equation}
Consider the term
\begin{equation}
K_6
=
\varepsilon
\sum_{j=1}^{\infty}
\int_0^t
\Bigl\|
\nabla\bigl(
G_j(u_n^\varepsilon)-G_j(v_n^\varepsilon)
\bigr)
\Bigr\|_H^2\,ds .
\end{equation}

Since
\[
\Bigl\|
\nabla\bigl(
G_j(u_n^\varepsilon)-G_j(v_n^\varepsilon)
\bigr)
\Bigr\|_H^2
\le
C_h\|\nabla d_n\|_H^2 ,
\]
hence, we obtain
\begin{equation}
\begin{aligned}
K_6
&\le
\varepsilon C_h
\int_0^t
\|\nabla d_n(s)\|_H^2\,ds .
\end{aligned}
\end{equation}
Collecting all estimates, we obtain
\begin{equation}
\begin{aligned}
\|\nabla d_n(t)\|_H^2
&\le \|\nabla d_n(0)\|_H^2
+ \varepsilon C_h T \\
&\quad + \int_0^t \bigl( \varepsilon C_h + \varepsilon C\beta_3 \bigr) \|\nabla d_n\|_H^2 \, ds \\
&\quad + C\varepsilon \int_0^t \Bigl( 1 + \frac{\beta_3}{\beta_1} \|u_n^\varepsilon\|_{H^1}^6 + \|v_n^\varepsilon\|_{H^2}^2 + \|u_n^\varepsilon\|_{H^1}^2 \\
&\qquad\qquad + \frac{\beta_4^2}{\beta_1} \|u_n^\varepsilon\|_{H^2}^2 \|u_n^\varepsilon\|_{H^1} + \beta_1 \|\Delta v_n^\varepsilon\|_{H}^2 + \beta_2 \|\nabla v_n^\varepsilon\|_{H}^2 \Bigr) \, ds \\
&\quad + 2\sqrt{\varepsilon} \sum_{j=1}^{\infty} \int_0^t \Bigl\langle \nabla d_n, \nabla\bigl( G_j(u_n^\varepsilon) - G_j(v_n^\varepsilon) \bigr) \Bigr\rangle_H \, dW_j(s).
\end{aligned}
\end{equation}
For \(M>0\), define the stopping times
\[
\sigma_{M}^{1,\varepsilon}
=
\inf\Biggl\{
t\ge0:
\|v_n^\varepsilon(t)\|_{H^2}^{2}>M
\Biggr\},
\]
and
\[
\sigma_{M}^{2,\varepsilon}
=
\inf\Biggl\{
t\ge0:
\int_0^t \|u_n^\varepsilon(s)\|_{H^2}^{2}\,ds>M,
\ \text{or}\
\|u_n^\varepsilon(t)\|_{H^1}^{2}>M
\Biggr\}.
\]
Let
\[
\sigma_{M}^{\varepsilon}
=
\sigma_{M}^{1,\varepsilon}
\wedge
\sigma_{M}^{2,\varepsilon}.
\]

According to all the above estimates, it follows that
\[
\begin{aligned}
& \sup_{0\le s\le t\wedge \sigma_{M}^{\varepsilon}} \|\nabla d_n(s)\|_{H}^{2}\le \varepsilon \int_0^{t\wedge \sigma_{M}^{\varepsilon}} \bigl( C_h + C\beta_3 \bigr) \|\nabla d_n(s)\|_{H}^{2} \, ds + \sup_{0\le s\le t\wedge \sigma_{M}^{\varepsilon}} |K_8(s)| + \varepsilon C_{h,M}^{T,\beta},
\end{aligned}
\]
where
\[
\begin{aligned}
C_{h,M}^{T,\beta}
:=&\; C_h T M^2 + C \Biggl[ T + \frac{\beta_3}{\beta_1} M^{6} + T M^2 + T M + \frac{\beta_4^{2}}{\beta_1} M^{3} + \beta_1 M^2 + \beta_2 T M^2 \Biggr]\,\,\text{and } C_h^{\beta_3}:= C_h + C\beta_3.
\end{aligned}
\]
Applying Gronwall's inequality, we obtain
\[
\begin{aligned}
& \sup_{0\le s\le t\wedge \sigma_{M}^{\varepsilon}} \|\nabla d_n(s)\|_{H}^{2} \\
&\le \Biggl[ \sup_{0\le s\le t\wedge \sigma_{M}^{\varepsilon}} |K_8(s)| + \varepsilon C_{h,M}^{T,\beta} \Biggr] \exp \Bigl\{ \varepsilon C_h^{\beta_3} T \Bigr\}.
\end{aligned}
\]

Now consider
\[
K_8(t) = 2\sqrt{\varepsilon} \sum_{j=1}^{\infty} \int_0^t \Bigl\langle \nabla d_n, \nabla\bigl( G_j(u_n^\varepsilon) - G_j(v_n^\varepsilon) \bigr) \Bigr\rangle_H \, dW_j(s).
\]

Since \(G_j\) is linear, we have
\[
G_j(u_n^\varepsilon) - G_j(v_n^\varepsilon) = -d_n \times h_j .
\]
Using
\[
\sum_{j=1}^{\infty}
\int_0^{T\wedge \sigma_{M}^{\varepsilon}}
\Bigl| \Bigl\langle \nabla(d_n(s) \times h_j), \nabla d_n(s) \Bigr\rangle_H \Bigr|^{2} \, ds
\le C_{h,M} \int_0^{T\wedge \sigma_{M}^{\varepsilon}} \|\nabla d_n(s)\|_{H}^{2} \, ds,
\]
the Burkholder--Davis--Gundy inequality yields
\[
\begin{aligned}
& \mathbb{E} \Biggl[ \sup_{0\le t\le T\wedge \sigma_{M}^{\varepsilon}} |K_8(t)|^{p} \Biggr] \\
&\le C_p \varepsilon^{p/2} \mathbb{E} \Biggl[ \Biggl( \sum_{j=1}^{\infty} \int_0^{T\wedge \sigma_{M}^{\varepsilon}} \Bigl| \Bigl\langle \nabla(d_n(s) \times h_j), \nabla d_n(s) \Bigr\rangle_H \Bigr|^{2} \, ds \Biggr)^{p/2} \Biggr] \\
&\le C_p \varepsilon^{p/2} \mathbb{E} \Biggl[ \Biggl( C_{h,M} \int_0^{T\wedge \sigma_{M}^{\varepsilon}} \|\nabla d_n(s)\|_{H}^{2} \, ds \Biggr)^{p/2} \Biggr].
\end{aligned}
\]
Hence, we obtain
\[
\begin{aligned}
\Biggl( \mathbb{E} \Bigl[ \sup_{0\le s\le T\wedge \sigma_{M}^{\varepsilon}} \|\nabla d_n(s)\|_{H}^{2p} \Bigr] \Biggr)^{\frac{2}{p}}& \le 2 \bigl( \varepsilon C_{h,M}^{T,\beta} \bigr)^2 \exp \Bigl\{ 2\varepsilon C_h^{\beta_3} T \Bigr\} \\
&\quad\qquad + \exp \Bigl\{ 2\varepsilon C_h^{\beta_3} T \Bigr\} C_{h,M}^{2} \, p \varepsilon \int_0^{T\wedge \sigma_{M}^{\varepsilon}} \Biggl( \mathbb{E} \bigl[ \|\nabla d_n(s)\|_{H}^{2p} \bigr] \Biggr)^{\frac{2}{p}} \, ds .
\end{aligned}
\]
Gronwall's inequality yields
\[
\begin{aligned}
& \Biggl( \mathbb{E} \Bigl[ \sup_{0\le s\le T\wedge \sigma_{M}^{\varepsilon}} \|\nabla d_n(s)\|_{H}^{2p} \Bigr] \Biggr)^{\frac{2}{p}}\le 2 \bigl( \varepsilon C_{h,M}^{T,\beta} \bigr)^2 \exp \Bigl\{ 2\varepsilon C_h^{\beta_3} T \Bigr\}\exp \Biggl\{ \exp \Bigl( 2\varepsilon C_h^{\beta_3} T \Bigr) C_{h,M}^{2} \, p \varepsilon T \Biggr\}.
\end{aligned}
\]
Taking \(p = \frac{2}{\varepsilon}\), and proceeding exactly as in the proof of Lemma~\ref{lemma3}, we obtain
\[
\begin{aligned}
& \varepsilon \log \mathbb{P} \Biggl( \sup_{0\le s\le T\wedge \sigma_{M}^{\varepsilon}} \|\nabla d_n(s)\|_{H}^{2} > \delta \Biggr) \longrightarrow -\infty, \qquad \varepsilon \to 0 .
\end{aligned}
\]
For any \(K > 0\), by Lemma~\ref{lemma1}, \cite[Lemma 5.2]{ROCKNER2012716}, and the above estimates, one can choose \(M > 0\) sufficiently large such that, for every \(0 < \varepsilon \le 1\),
\[
\mathbb{P} \Bigl( \|u_n^\varepsilon\|_{H^1}^{H^2}(T) > M \Bigr) \le e^{-K/\varepsilon},
\]
and
\[
\mathbb{P} \Biggl( \sup_{0 \le t \le T} \|v_n^\varepsilon(t)\|_{H^2}^{2} > M \Biggr) \le e^{-K/\varepsilon}.
\]

Furthermore, the previous estimate implies that there exists \(0 < \varepsilon_0 < 1\) such that, for all \(\varepsilon \le \varepsilon_0\),
\[
\begin{aligned}
& \varepsilon \log \mathbb{P} \Biggl( \sup_{0 \le s \le T} \|\nabla u_n^\varepsilon(s) - \nabla v_n^\varepsilon(s)\|_{H}^{2} > \delta,  \|u_n^\varepsilon\|_{H^1}^{H^2}(T) \le M, \quad \sup_{0 \le t \le T} \|v_n^\varepsilon(t)\|_{H^2}^{2} \le M \Biggr) \\
&\le \varepsilon \log \mathbb{P} \Biggl( \sup_{0 \le s \le T \wedge \sigma_{M}^{\varepsilon}} \|\nabla u_n^\varepsilon(s) - \nabla v_n^\varepsilon(s)\|_{H}^{2} > \delta \Biggr) \le -K .
\end{aligned}
\]
Consequently, for all $\varepsilon\le\varepsilon_0$,
\[
\mathbb P\Bigl(\sup_{0\le t\le T}\|\nabla u_n^\varepsilon(t)-\nabla v_n^\varepsilon(t)\|_H^2>\delta\Bigr)
\le 3e^{-K/\varepsilon}.
\]
On the other hand, Lemma~\ref{lemma4} yields
\[
\mathbb P\Bigl(\sup_{0\le t\le T}\|u_n^\varepsilon(t)-v_n^\varepsilon(t)\|_H^2>\delta\Bigr)
\le 3e^{-K/\varepsilon}.
\]
Therefore,
\[
\mathbb P\Bigl(\sup_{0\le t\le T}\|u_n^\varepsilon(t)-v_n^\varepsilon(t)\|_{H^1}^2>2\delta\Bigr)
\le 6e^{-K/\varepsilon},
\]
where we have used the identity
\(
\|w\|_{H^1}^2=\|w\|_H^2+\|\nabla w\|_H^2
\). Since \(K>0\) is arbitrary, the proof is complete.
\end{proof}

\begin{lemma}\label{lemma6}
For any \(\delta>0\), we have
\[
\lim_{n\to\infty}
\sup_{0<\varepsilon\le1}
\varepsilon
\log
\mathbb{P}
\Biggl(
\sup_{0\le t\le T}
\|v^\varepsilon(t)-v_n^\varepsilon(t)\|_{H^1}^{2}
>2\delta
\Biggr)
=
-\infty .
\]
\end{lemma}

\begin{proof}
The proof is completely analogous to that of Lemma~\ref{lemma2} and Lemma~\ref{lemma3}. Hence, we omit the details.
\end{proof}

\begin{proof}[Proof of Theorem~\ref{main_theorem}]

We now verify \eqref{exp_equiv} to complete the proof of the theorem. Let \(K>0\) be arbitrary. By Lemma~\ref{lemma2} and Lemma~\ref{lemma6}, there exists a positive integer \(N_{0}\) such that, for all \(0<\varepsilon\le1\),
\[
\mathbb P\Bigl(
\sup_{0\le t\le T}
\|u^\varepsilon(t)-u_{N_0}^\varepsilon(t)\|_{H^1}^{2}
>2\delta
\Bigr)
\le 6e^{-K/\varepsilon},
\]
and
\[
\mathbb P\Bigl(
\sup_{0\le t\le T}
\|v^\varepsilon(t)-v_{N_0}^\varepsilon(t)\|_{H^1}^{2}
>2\delta
\Bigr)
\le 6e^{-K/\varepsilon}.
\]

On the other hand, for this \(N_0\), Lemma~\ref{lemma5} implies that there exists \(0<\varepsilon_0<1\) such that, for every \(0<\varepsilon\le\varepsilon_0\),
\[
\mathbb P\Bigl(
\sup_{0\le t\le T}
\|u_{N_0}^\varepsilon(t)-v_{N_0}^\varepsilon(t)\|_{H^1}^{2}
>2\delta
\Bigr)
\le 6e^{-K/\varepsilon}.
\]

Therefore, for all \(0<\varepsilon\le\varepsilon_0\),
\[
\begin{aligned}
&\mathbb P\Bigl(
\sup_{0\le t\le T}
\|u^\varepsilon(t)-v^\varepsilon(t)\|_{H^1}^{2}
>18\delta
\Bigr) \\
&\le
\mathbb P\Bigl(
\sup_{0\le t\le T}
\|u^\varepsilon(t)-u_{N_0}^\varepsilon(t)\|_{H^1}^{2}
>2\delta
\Bigr)
+
\mathbb P\Bigl(
\sup_{0\le t\le T}
\|u_{N_0}^\varepsilon(t)-v_{N_0}^\varepsilon(t)\|_{H^1}^{2}
>2\delta
\Bigr) \\
&\qquad +
\mathbb P\Bigl(
\sup_{0\le t\le T}
\|v^\varepsilon(t)-v_{N_0}^\varepsilon(t)\|_{H^1}^{2}
>2\delta
\Bigr) \\
&\le 18e^{-K/\varepsilon}.
\end{aligned}
\]
Since \(K>0\) and \(\delta>0\) are arbitrary, \eqref{exp_equiv} follows.
\end{proof}

\paragraph{Data Availability.}
No data were generated or analyzed during this study.

\paragraph{Competing Interests.}
The author declares that there are no competing interests.
\printbibliography
\end{document}